\documentclass{article}
\usepackage[margin = 1.5in]{geometry}

\usepackage{times}
\usepackage{bm}
\usepackage{natbib}
\usepackage{amsthm, amsmath, amssymb, mathrsfs, graphicx, multirow, url, subfigure}
\usepackage{textpos}
\usepackage{ifthen}
\usepackage[usenames]{color}

\theoremstyle{plain}
\newtheorem{theorem}{Theorem}
\newtheorem{lemma}[theorem]{Lemma}

\theoremstyle{definition}

\newtheorem{example}{Example}[section]

\newcommand{\F}{\mathbb{F}}

\newcommand{\I}{\mathscr{I}}

\newcommand{\OO}{\mathbb{O}}
\newcommand{\RR}{\mathbb{R}}
\renewcommand{\SS}{\mathbb{S}}
\newcommand{\TT}{\mathbb{T}}

\newcommand{\nmin}{n_{\mathrm{min}}}
\newcommand{\Psitilde}{\bar{\Psi}}
\newcommand{\Ftilde}{\bar{F}}

\newcommand{\E}{\mathsf{E}}

\newcommand{\cov}{\mathsf{Cov}}
\newcommand{\prob}{\mathsf{P}}
\DeclareMathOperator{\tr}{tr}

\newcommand{\be}{{\sf Be}}
\newcommand{\nm}{{\sf N}}
\newcommand{\dpp}{{\sf DP}}
\newcommand{\dpm}{{\sf DPM}}
\newcommand{\stt}{{\sf t}}
\newcommand{\wish}{{\sf W}}
\newcommand{\iwish}{{\sf IW}}
\newcommand{\fdist}{{\sf F}}

\newcommand{\eps}{\varepsilon}
\renewcommand{\phi}{\varphi}

\newcommand{\iid}{\overset{\text{\tiny iid}}{\sim}}

\newcommand{\jth}{j\text{th}}

\newcommand{\imp}{\text{imp}}

\newcommand{\h}[1]{{ #1}}

\begin{document}

\title{Bayesian test of normality versus a Dirichlet process mixture alternative}

\author{Surya T. Tokdar \\ Department of Statistical Science, Duke University, \\ Durham, North Carolina 27708, U.S.A. \and Ryan Martin\\ Department of Statistics, North Carolina State University\\ Raleigh, North Carolina 27607, U.S.A.}

\maketitle

\begin{abstract}
We propose a Bayesian test of normality for univariate or multivariate data against alternative nonparametric models characterized by Dirichlet process mixture distributions.  The alternative models are based on the principles of embedding and predictive matching.  They can be interpreted to offer random granulation of a normal distribution into a mixture of normals with mixture components occupying a smaller volume the farther they are from the distribution center. A scalar parametrization based on latent clustering is used to cover an entire spectrum of separation between the normal distributions and the alternative models.  An efficient sequential importance sampler is developed to calculate Bayes factors. Simulations indicate the proposed test can detect non-normality without favoring the nonparametric alternative when normality holds. 

\end{abstract}

\section{Introduction}
\label{S:intro}

Professor Jayanta K.~Ghosh has left behind a lasting legacy in many areas of statistics research. Three prominent such areas are (a) formal/objective Bayes, (b) nonparametric models, and (c) model comparison/selection. One interesting problem that lies at the intersection of these three areas is the question of formally assessing the fit of a parametric model against nonparametric alternatives \citep{tokdar2010bayesian}. In this paper we attempt to {\it settle this question} when assessing the fit of univariate or multivariate normal models. Although several goodness-of-fit tests exist for assessing normality with the usual emphasis on the null model \citep{cardoso2010multivariate,aldor2013power, voinov2016new}, currently lacking in the literature is a satisfactory formal Bayesian solution that also places equal emphasis on the alternatives, i.e., on the possible modes of departure from normality. 

The availability of nonparametric alternatives means  parametric models are no longer indispensable.  At the same time, when appropriate, they provide considerable simplification and more penetrative inference compared to a nonparametric model.  But it is important that they are first tested for appropriateness.  In some cases parametric models directly represent a precise scientific hypothesis, such as the Gaussianity of the Cosmic Microwave Background \citep[e.g.,][]{barreiro2007gaussian}.  In many other cases parametric models provide the clearest modeling framework to embed a scientific hypothesis. For example, the hypothesis of flipping between stationary states by a neuron in response to multiple stimuli \citep{abeles1995cortical, jones2007natural} is most easily tested when stationary states are described by identifiable parametric models.  When additional data are available from single stimulus trials, the parametric model can be and should be tested first.

\h{
A formal Bayesian assessment of model fit is challenging for several reasons. It requires a complete specification of an alternative model and involves the difficult calculation of the Bayes factor: the ratio between the marginal data likelihood scores under the null and alternative specifications. In assessing the fit of a parametric model it is nearly impossible to specify a broad alternative model without using subjective knowledge. But progress has been made in this direction with advances in nonparametric Bayes methodology; see \cite{berger2001bayesian, verdinelli1998bayesian, florens1996bayesian, carota1996bayes}. These authors have advocated for a certain level of formalism in choosing a nonparametric alternative that is not only an attractive model for data analysis, but can also be viewed as an extension of the null parametric model that remains non-informative with respect to the parameters of the null model.

In particular, \cite{berger2001bayesian} advocate choosing a nonparametric alternative that is {\it balanced} against the parametric null in the sense of {\it embedding} and {\it predictive matching} properties. Loosely speaking, embedding refers to the property that the alternative model space can be partitioned in such a way that each partition represents an unbiased relaxation of one and only one element of the parametric model space. Given such an embedding, the same parameter $\theta$ that indexes the elements of the null model could be used to index the partitions of the alternative model, and, a common, noninformative prior may be used on $\theta$ under either specification. Predictive matching formalizes this correspondence in a strong technical way, demanding that the Bayes factor remain neutral between the null and the alternative until one has accumulated sufficient amount of data. That ``sufficient amount'' is taken to be the minimum sample size needed to get proper posteriors on $\theta$ under both specifications.}

\h{With this formalism in mind,} we pursue a new Bayesian method for assessing the fit of the normal model to univariate or multivariate data. Currently there are two fully developed approaches toward assessing the fit of the normal model, the Gaussian process approach of \citet{verdinelli1998bayesian} and the Polya tree approach of \citet{berger2001bayesian}. We propose a new alternative model based on a Dirichlet process location-scale mixture of normals \citep{lo1984class} with several advantages over these existing techniques.

The Gaussian process approach is difficult to compute with and does not allow for embedding and predictive matching.  The P\'olya tree approach is easy to work with for univariate data.  But its reliance on partition-based computing does not scale well with data dimension.  Moreover, a Polya tree distribution is a model for densities that are nowhere differentiable \citep{choudhuri2005bayesian}.  This may lead to inefficient estimation under the alternative \citep{van2008rates, castillo2008lower} which, in turn, may lead to a sub-optimal detection of non-normality.  Our simulation study provides evidence supporting this claim.

In contrast, our Dirichlet process mixture of normals model is in itself an attractive model for estimating a smooth density. Dirichlet process mixture of normals have been well studied in the literature and are known to be easy to compute with, often via efficient Gibbs sampling or its variations \citep{escobar1995bayesian, maceachern1998estimating, maceachern1998computational, neal2000markov}, and are known also to possess optimal, adaptive convergence rates in a variety of density estimation applications \citep{ghosal2001entropies, ghosal2007posterior, shen2013adaptive}.

In formulating a Dirichlet process mixture of normals, we diverge slightly from standard constructions and use a {\it normal-multivariate-beta} base measure (Section \ref{SS:model}); \h{see \citet{griffin2010default} for a related formulation}. This helps us construct a collection of Dirichlet process mixture of normals priors which are mapped one-to-one to the collection of all normal densities. Each element of the alternative model may be understood as a random granulation of the corresponding normal density into a mixture of normals with the volume of a mixture component negatively correlated with its lateral shift from the center. Our alternative model is parametrized by a single scalar parameter, the {\it precision} parameter of the underlying Dirichlet process. The precision parameter controls the extent of granulation, i.e., latent clustering, which is key in determining the separation between the null and the alternative. Other potential model parameters, such those controlling the extent of lateral shifts of the mixture components, are carefully mapped to the precision parameter to avoid identifiability problems when precision is close to zero or infinity.

Despite the slightly different formulation, our Dirichlet process mixture of normals model is amenable to Gibbs sampling for posterior computation and to sequential imputation \citep{liu1996nonparametric} and posterior ordinate calculation \citep{basu2003marginal} for Bayes factor computation. In Section~\ref{S:computation}, we propose a reasonably efficient algorithm for Bayes factor computation by adapting Liu's sequential imputation technique to our formulation and augmenting it with importance sampling to deal with additional parameters that are not part of the Dirichlet process mixing distribution. This algorithm is demonstrated to perform much better than two reasonable adaptations of the posterior ordinate approach \citep{basu2003marginal}. \h{For analyzing multivariate data, we propose an extension of this algorithm that uses a Rao--Blackwellized parameter augmentation technique, borrowing ideas from sequential Monte Carlo.}  

Section~\ref{SS:comp} presents a simulation study of the proposed method's Type I and II error probabilities within a frequentist setting of hypothesis testing. In a univariate setting, the resulting test is found to offer moderate to large improvements in power for a given size when compared to a test based on the P\'olya tree approach \citep{berger2001bayesian} and the classical Anderson--Darling test. In Section \ref{SS:consistency} we address the important issue of Bayes factor consistency \citep{tokdar2010bayesian} which refers to the desirable frequentist property: Bayes factor goes to $\infty$ under the null and goes to $0$ under the alternative asymptotically as sample size grows to infinity. We do not consider a full theoretical study of Bayes factor consistency, due to severe technical challenges, but we provide a large sample simulation study with sample size up to 5000. Our simulations give strong evidence of consistency under the null. Consistency under the alternative is well expected for Dirichlet process mixture models \citep[][Section 4]{tokdar2010bayesian}.

\section{A Dirichlet mixture of normals method for testing normality}
\subsection{Formalization of the testing problem}
\label{SS:formulation}

Consider data $X_{1:n} = (X_1, \ldots, X_n)$ where $X_i \in \RR^p$, $i = 1, \ldots, n$ are modeled as $n$ independent draws from an unknown common probability distribution $F$. Let $F_{\mu, \sigma}$ denote a $p$-variate normal distribution with mean $\mu$ and covariance matrix $\sigma\sigma^\top$ in Cholesky decomposition form and define $\F_0 = \{F_{\mu,\sigma}: \mu \in \RR^p, \sigma \in \TT_p\}$ where $\TT_p$ is the set of all $p \times p$ lower-triangular matrices with positive diagonal elements. Our goal is to test $H_0: F \in \F_0$.  

Unlike classical goodness-of-fit tests, any Bayesian approach to this testing problem requires two additional model ingredients.  First, the null model requires a possibly improper prior distribution $\pi_0$ on $\RR^p \times \TT_p$.  Second, an  alternative model $H_1: F \in \F_1$ is required, along with a prior $\Pi_1$ on $\F_1$.  For a non-subjective treatment, it is natural to choose $\F_1$ an infinite-dimensional subset of probability measures on $\RR^p$, and $\Pi_1$ a probability measure supported on $\F_1$.  Once the priors $\pi_0$ and $\Pi_1$ are specified, one can report the Bayes factor 
\begin{equation}
\label{eq:bayes.factor}
B = \frac{ \int_{\RR^p \times \TT_p} \bigl\{ \prod_{i=1}^n dF_{\mu,\sigma}(x_i) \bigr\} \,d\pi_0(\mu,\sigma) }{ \int_{\F_1} \bigl\{ \prod_{i=1}^n dF(x_i) \bigr\} \,d\Pi_1(F) }
\end{equation}
as a measure of evidence against $H_0$ when $X_{1:n} = x_{1:n}$ are observed. Small $B$ indicates the parametric model provides an unsatisfactory fit to the data. Refer to \citet{kass1995bayes} for more on the Bayes factor and its interpretation.

\subsection{Local alternative, null embedding and a new Dirichlet process mixture}
\label{SS:model}

For a non-subjective test, \citet{berger2001bayesian, verdinelli1998bayesian, florens1996bayesian} and \citet{carota1996bayes} stress on the importance of maintaining balance between the null model and the non-parametric alternative. All these authors recommend specifying $\Pi_1$ as $\int \Pi_{\mu, \sigma} d\pi_1(\mu, \sigma)$ a mixture of local alternatives $\Pi_{\mu, \sigma}$ mapped one-to-one to the elements $F_{\mu, \sigma}$ of the null model. Most of these authors require this mapping to be given by embedding the null element as the mean of the local alternative: $\int F d\Pi_{\mu, \sigma}(F) = F_{\mu, \sigma}$ for every $(\mu, \sigma)$. This is difficult to achieve with the commonly used Dirichlet process mixture of normals \citep{escobar1995bayesian} that use a normal-inverse-Wishart base measure. We offer the following modification along the lines of \citet{griffin2010default}.

Let $\SS_p$ be the space of $p \times p$ symmetric positive definite matrices with all $p$ eigenvalues in $(0,1)$.  For scalars $\omega_1$ and $\omega_2$ greater than $(p-1)/2$, let $\be(\omega_1,\omega_2)$ denote the multivariate beta distribution on $\SS_p$ \citep[][Chap.~3.3]{muirhead2005aspects} having density 
\begin{equation}
\label{eq:mbeta}
\be(v \mid \omega_1,\omega_2) = a_p(\omega_1,\omega_2)(\det v)^{\omega_1 - (p+1)/2} \{\det(I_p - v)\}^{\omega_2 - (p+1)/2}, 
\end{equation}
where $I_p$ is the $p \times p$ identity matrix and $a_p(\omega_1, \omega_2) = \Gamma_p(\omega_1 + \omega_2) / \Gamma_p(\omega_1) \Gamma_p(\omega_2)$, with $\Gamma_p$
the $p$-variate gamma function.  Write $\Psi$ for the probability measure on $\RR^p \times \SS_p$ given by the law of $(U,V)$, where $V \sim \be(\omega_1,\omega_2)$ and $U \mid V \sim \nm(0,I_p-V)$.  This law is well-defined, since $I_p-V \in \SS_p$ with probability~1.  

Let $\dpp(\alpha,\Psi)$ denote the Dirichlet process distribution with precision $\alpha > 0$ and base measure $\Psi$ from above \citep{ferguson1973bayesian}.  Recall that $\Psitilde \sim \dpp(\alpha,\Psi)$ means that for any positive integer $k$ and any measurable partition $B_1,\ldots,B_k$ of $\RR^p \times \SS_p$, the probability vector $\{\Psitilde(B_1),\ldots,\Psitilde(B_k)\}$ has a $k$-dimensional Dirichlet distribution with parameters $\{\alpha \Psi(B_1),\ldots,\alpha \Psi(B_k)\}$.  For any $(\mu,\sigma)$, let $\dpm_{\mu,\sigma}(\alpha,\Psi)$ denote the distribution of the random probability measure 
\begin{equation}
\label{eq:dpm}
\Ftilde_{\mu,\sigma} = \int \nm(\mu + \sigma u, \sigma v \sigma^\top) \,d\Psitilde(u,v), \quad \text{where} \quad \Psitilde \sim \dpp(\alpha,\Psi);
\end{equation}
then we have the following result.
\begin{theorem}
\label{thm:embedding}
For any $(\mu,\sigma)$ and any $\alpha$, the mean of $\dpm_{\mu,\sigma}(\alpha,\Psi)$ is $\nm(\mu,\sigma\sigma^\top)$.  
\end{theorem}

\begin{proof}
For an $\bar F_{\mu, \sigma}$ as in \eqref{eq:dpm}, its expectation is simply $\int \nm(\mu + \sigma u, \sigma v \sigma^\top) \,d\Psi(u, v) = \int \{\int \nm(\mu + \sigma u, \sigma v\sigma^\top) \, d\nm(u \mid 0, I_p - v)\} \, d\be(v \mid \omega_1, \omega_2)$ by definition of $\Psi$. But the inner integral always equals $\nm(\mu, \sigma\sigma^\top)$ by the well-known Gaussian convolution identity.
\end{proof}

We choose $\dpm_{\mu, \sigma}(\alpha, \Psi)$ as the local alternative $\Pi_{\mu, \sigma}$ to $F_{\mu, \sigma}$, with Theorem \ref{thm:embedding} ensuring local embedding. It is more convenient to write our null and alternative models in the following hierarchical manner.
\begin{align}
H_0: & \; X_{1:n} \mid (\mu,\sigma) \iid F_{\mu,\sigma}, \quad (\mu,\sigma) \sim \pi_0 \label{eq:null} \\
H_1: & \; X_{1:n} \mid (\Ftilde_{\mu,\sigma}, \mu, \sigma) \iid \Ftilde_{\mu,\sigma}, \quad \Ftilde_{\mu,\sigma} \mid (\mu,\sigma) \sim \dpm_{\mu,\sigma}(\alpha,\Psi), \quad (\mu,\sigma) \sim \pi_1; \label{eq:alt}
\end{align}
The choice of $\pi_0$, $\pi_1$ will be discussed in Section \ref{SS:pies}.

\subsection{Understanding local alternative as a random granulation}
\label{SS:clusters}

For any space $S$ and an $s \in S$, let $\langle s\rangle$ denote the degenerate probability distribution on $S$ with point mass at $s$. Due to the stick-breaking representation of a Dirichlet process \citep{sethuraman1994constructive} a random $\Psitilde \sim \dpp(\alpha, \Psi)$ can be written as 
\begin{equation}
\label{eq:stick.break}
\Psitilde = \sum_{h \geq 1} q_h \langle (U_h, V_h) \rangle, 
\end{equation}
where $(U_h, V_h)$, $h \geq 1$, are independently draws from $\Psi$, $q_h = \beta_h \prod_{j < h} (1-\beta_j)$ and $\beta_h$, $h \geq 1$, are independent draws from a univariate $\be(1,\alpha)$ distribution.  The vector $q_{1:\infty} = (q_1, q_2, \ldots)$ satisfies $q_h \geq 0$ and $\sum_h q_h = 1$, with probability~1.  Consequently, given $(\mu,\sigma)$, a draw from the local Dirichlet process mixture alternative $\dpm_{\mu,\sigma}(\alpha,\Psi)$ can be written as
\begin{equation}
\label{eq:local.draw}
\Ftilde_{\mu,\sigma} = \sum_{h \geq 1} q_h \nm(\mu + \sigma U_h, \sigma V_h \sigma^\top),
\end{equation}
with $(q_h, U_h, V_h)$, $h \geq 1$, as described above.  Therefore, given $(\mu,\sigma)$, the local alternative $X_{1:n} \iid \Ftilde_{\mu,\sigma}$ is equivalent to saying that the $X_i$'s are independently with distribution $\nm(\mu + \sigma U_{h_i}, \sigma V_{h_i} \sigma^\top)$ where the $h_i$'s are randomly drawn labels with $\prob(h_i = h)  = q_h$.  Ties among the $h_i$'s partition the data $X_{1:n}$ into clusters, where the $X_i$'s in a cluster are independent $\nm(\mu+\sigma U_h, \sigma V_h \sigma^\top)$ observations, with $(U,V) \sim \Psi$.  The center of this cluster is at a $\sigma U$ shift from the center $\mu$ of the null element $\nm(\mu,\sigma\sigma^\top)$ and occupies a $(\det V)^{1/2} \in (0,1)$ fraction of the corresponding volume.  Theorem~\ref{thm:neg.cov} shows that the magnitude $(U^\top U)^{1/2}$ of the shift, relative to $\sigma$, is stochastically inversely related to the volume fraction $(\det V)^{1/2}$.

\begin{theorem}
\label{thm:neg.cov}
If $(U,V) \sim \Psi$, then $\cov(U^\top U, \det V) \leq 0$.  
\end{theorem}

Therefore, for given $(\mu,\sigma)$, $\Ftilde_{\mu,\sigma}$ in \eqref{eq:local.draw} can be seen as local granulation of a population of fine particles evenly distributed according to $\nm(\mu,\sigma\sigma^\top)$.  The local granulations forms clusters with bell-shaped curves, each occupying only a fraction of the total volume of the population.  The further the cluster center is from the original $\nm(\mu,\sigma\sigma^\top)$ population center, the smaller the cluster size is likely to be.  

\subsection{Separation between null and alternative and the choice of $\omega_1,\omega_2$}
\label{SS:omega}

All three parameters $\alpha$, $\omega_1$ and $\omega_2$ contribute to making the alternative look different from the null. The precision parameter $\alpha$ controls the degree of clustering, i.e., the prevalence of ties among the cluster labels $h_i$ introduced above \citep[see][Chap.~3]{ghosh2003bayesian}. Base measure parameters $\omega_1$, $\omega_2$ control lateral shifts and relative volumes of the cluster components. We argue that it is important to link the specification of $(\omega_1, \omega_2)$ to that of $\alpha$, because otherwise the alternative model may acquire strange features that go against the notion of local embedding. 

If we fix $\omega_1, \omega_2$, thus fixing the base measure $\Psi$, and let $\alpha \to 0$ then $\dpp(\alpha,\Psi)$ converges weakly to the law of the random degenerate distribution $\langle(U, V)\rangle$ with $(U, V)$ drawn from $\Psi$ \citep[e.g.,][Chapter 3.2]{ghosh2003bayesian}. Hence, for $\bar F \sim \dpm_{\mu,\sigma}(\alpha,\Psi)$, the limiting law of $\bar F$ as $\alpha \to 0$ can be described as: $\bar F = \nm(\mu + \sigma U, \sigma V \sigma^\top)$, $(U, V) \sim \Psi$. In the limit, separation between the overall null and the overall alternative models vanishes, as both concentrate on the normal distributions. However, a positive difference remains between the local alternative $\dpm_{\mu, \sigma}(\alpha, \Psi)$ and $F_{\mu, \sigma}$. This discrepancy between global and local separations goes away if we make $\omega_1, \omega_2$ depend on $\alpha$ so that $\omega_1 / (\omega_1 + \omega_2) \to 1$ as $\alpha \to 0$. With such a choice of $(\omega_1, \omega_2)$, $\Psi$ converges to $\langle (0,I_p) \rangle$ as $\alpha \to 0$, and consequently the local alternative $\dpm_{\mu,\sigma}(\alpha,\Psi)$ collapses onto $\langle F_{\mu,\sigma} \rangle$ in the limit.

Irrespective of the choice of $(\omega_1, \omega_2)$, the local and global separations between the null and the alternative vanish as $\alpha \to \infty$. This is because, as $\alpha \to \infty$, $\dpp(\alpha,\Psi)$ converges to $\langle\Psi\rangle$ \citep[][Theorem~3.2.6]{ghosh2003bayesian}.  Consequently, for any $(\mu,\sigma)$, $\dpm_{\mu,\sigma}(\alpha,\Psi)$ converges to $\langle F_{\mu,\sigma}\rangle$ since $\int \nm(\mu + \sigma u, \sigma  v \sigma^\top) \, d\Psi(u, v) = \nm(\mu, \sigma\sigma^\top)$.  However the nature of this convergence depends on the limiting behavior of $\omega_1 / (\omega_1 + \omega_2)$. In particular, choosing $\omega_1 / (\omega_1 + \omega_2) \to 0$ as $\alpha \to \infty$ brings in some additional, useful flexibility of the alternative model. For large values of $\alpha$, the stick-breaking representation \eqref{eq:stick.break} of an $\Ftilde_{\mu,\sigma} \sim \dpm_{\mu,\sigma}(\alpha,\Psi)$ does not contain any dominating $q_h$ and is thus made up of small contributions from many normal components. If in addition $\omega_1 / (\omega_1 + \omega_2)$ is close to 0, then all these components have tiny relative volumes, but together they resemble the shape of $\nm(\mu,\sigma\sigma^\top)$. Such a model allows detection of non-normal distributions that have an overall shape like a bell curve, but possess sharp local features.  

 Based on these two limit scenarios, we recommend mapping the choice of $\omega_1, \omega_2$ to that of $\alpha$ such that ${\omega_1}/{(\omega_1 + \omega_2)}$ converges to 1 as $\alpha \to 0$ and converges to 0 as $\alpha \to \infty$. An optimal choice of $\omega_1, \omega_2$ satisfying these limits remains an open question. We have carried out a limited simulation study with $\omega_1, \omega_2$ of the form: $\omega_1 = c + g(1/\alpha)$ and $\omega_2 = c + g(\alpha)$ for some $c \ge (p  - 1)/2$ and some monotone increasing function $g$. In our study (not reported) reasonable testing performance was obtained if we picked $g(x) = x^k$ where the power $k$ increased with dimension. In the  experiments reported in Sections \ref{S:illustrations} and \ref{S:simu} we use $c = k = (p + 1)/2$, that is, our specification of $(\omega_1, \omega_2)$ is
\begin{equation}
\label{eq:omega}
\omega_1 = \tfrac{p + 1}{2} + \alpha^{-\frac{p + 1}{2}} \quad \text{and} \quad \omega_2 = \tfrac{p + 1}{2} + \alpha^{\frac{p + 1}{2}}.
\end{equation}
As reported in Section \ref{SS:comp}, this choice of $(\omega_1, \omega_2)$ leads to a fairly accurate testing procedure. Figure \ref{fig:omegas} shows one random draw from $\dpm_{0, 1}(\alpha, \Psi)$ for the univariate case, with $\Psi$ determined as by \eqref{eq:omega} for three choices of $\alpha \in \{2^{-6}, 2^2, 2^{10}\}$. For small $\alpha$, there is little difference between the $\nm(0,1)$ and its local alternative. For large $\alpha$, there is an overall shape resemblance, but the alternative possesses sharp features. Broad shape differences are noticed for an intermediate $\alpha$ value. 

\begin{figure}
\centering
\includegraphics[scale = 0.8]{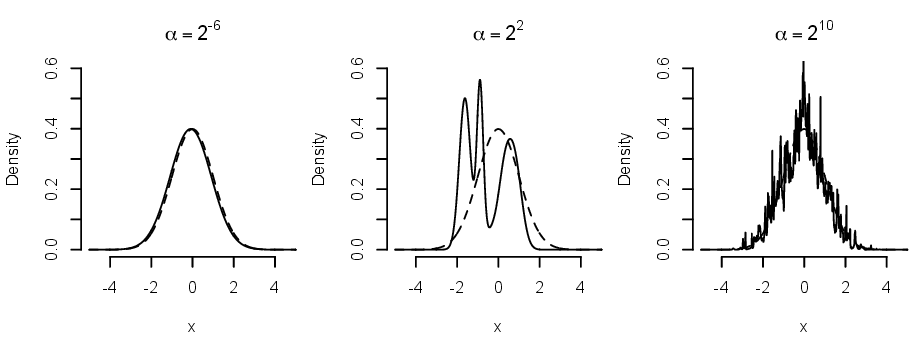}
\caption{Comparison between $\nm(0,1)$ and its local alternative $\dpm_{0,1}(\alpha, \Psi)$ for $p = 1$, with $\omega_1, \omega_2$ as chosen in \eqref{eq:omega}. One draw each (solid line)  from the local alternative for three choices of $\alpha$. Dashed line shows $\nm(0, 1)$ density.}
\label{fig:omegas}
\end{figure}

A concern over the coupling between $(\omega_1, \omega_2)$ and $\alpha$ is whether the alternative is ever allowed to separate from the null for intermediate $\alpha$ values. Although we do not have a theoretical result to resolve this issue, in all our numerical studies in Section \ref{S:illustrations}, except for when data are simulated from a normal distribution, the Bayes factor is found to attain very large magnitudes for a reasonably wide range of intermediate $\alpha$ values.  Because the alternative always embeds the null as its center, the only way it can concede so much ground to the null is by being fairly disperse around it.

\subsection{Predictive matching and choice of $\pi_0,\pi_1$}
\label{SS:pies}

For many parametric models $\mathbb{F}_0$, a default choice of $\pi_0$, usually improper, can be obtained through formal arguments such as invariance. A common choice for the normal model is the left Haar measure $\pi_L$ on $\RR^p \times \TT_p$, given by $d\pi_L(\mu,\sigma) = \prod_{j=1}^p \sigma_{jj}^{-j}\,d\mu\,d\sigma$, where $\sigma_{jj}$ is the $\jth$ diagonal element of $\sigma$. The corresponding prior on $(\mu, \Sigma = \sigma\sigma^\top)$ has the more familiar form: $\pi_L(\mu, \Sigma) =  (\det \Sigma)^{-(p+1)/2}$ and is also known as the independence Jeffreys' prior \citep{sun2007objective}.

In light of the null embedding property, it is tempting to choose $\pi_1 = \pi_0$ so that the elements of $\RR^p \times \TT_p$ are weighted the same under the null and alternative models.  \citet{berger2001bayesian} find this reasoning insufficient and argue that the choice $\pi_1 = \pi_0$ is partially justified whenever the predictive distribution of a hypothetical sample of size $\nmin$ is the same under the two models, where $\nmin$ is the minimal sample size needed to obtain a proper posterior for $(\mu,\sigma)$ under either model. They refer to this property as ``predictive matching''.

We show that \eqref{eq:null} and \eqref{eq:alt} have the predictive matching property with $\pi_0 = \pi_1 = \pi_L$. Toward this we present the following powerful result which gives a multivariate extension of a similar result in \citet{berger1998bayes}. We first need some notations and nomenclature. For any probability measure $F$ on $\RR^p$, let $F^{\times k}$ denote the $k$-fold product measure, i.e., $F^{\times k}$ is the probability law of $X_{1:k} = (X_1, \ldots, X_k)$ when the $X_i$'s are independent and identically distributed as $F$. Then $M_{\Pi,k} = \int F^{\times k} \,d\Pi(F)$ is the prior-predictive joint distribution of a sample from the model $X_{1:k} \sim F^{\times k}$, $F \sim \Pi$. If $F$ is almost surely absolutely continuous with respective to the Lebesgue measure, then $M_{\Pi, k}$ has a Lebesgue density $m_{\Pi, k}$ and $m_{\Pi, k}(x_{1:k})$ gives the marginal likelihood when data $x_{1:k}$ is observed for $X_{1:k}$.

A collection $\{\Pi_{\mu,\sigma}: (\mu,\sigma) \in \RR^p \times \TT_p\}$ where each $\Pi_{\mu, \sigma}$ is a probability measure on the space of probability measures on $\RR^p$, will be called a location-scale family if there is a random probability measure $F^\star$ on $\RR^p$ such that, for any $(\mu,\sigma)$, the law of the random measure $F_{\mu,\sigma}^\star$ defined as $dF_{\mu,\sigma}^\star(x) = |\det\sigma|^{-1} \,dF^\star(\sigma^{-1}(x-\mu))$ is precisely $\Pi_{\mu,\sigma}$.  A location-scale family will be called rotation-invariant if the random measures $F_{0,\eta'}^\star$ and $F^\star$ have the same law for any orthogonal matrix $\eta$.  Also, we shall call a location-scale family absolutely continuous if the characterizing $F^\star$ is absolutely continuous with respect to the Lebesgue measure with probability~1.  

\begin{theorem}
\label{thm:bpv}
Let $F \sim \Pi = \int \Pi_{\mu,\sigma} \,d\pi_L(\mu,\sigma)$ be a random probability measure on $\RR^p$, where $\{\Pi_{\mu,\sigma}: (\mu,\sigma) \in \RR^p \times \TT_p\}$ is an absolutely continuous, rotation-invariant, location-scale family, and $\pi_L$ is the left Haar measure on $\RR^p \times \TT_p$.  Then, for any $x_1,\ldots,x_{p+1}$ such that $\{\tilde x_j = x_j - x_{p+1} \in \RR^p: j =1,\ldots,p\}$ are linearly independent, 
\begin{equation}
\label{eq:pred}
m_{\Pi,p+1}(x_1,\ldots,x_{p+1}) = c_p^{-1} |\det \tilde x|^{-p}, 
\end{equation}
where $\tilde x$ is the $p \times p$ matrix with columns $\tilde x_1,\ldots,\tilde x_p$, and $c_p = 2^p \pi^{p^2/2} / \Gamma_p(p/2)$.  
\end{theorem}

In particular, for $p=1$, the minimum sample size is $\nmin=2$ and such a sample consists of two distinct observations, say, $x_1$ and $x_2$.  Then $\tilde x$ is a scalar, namely $x_1-x_2$, and $|\det \tilde x| = |x_1-x_2|$.  Also, a direct calculation gives $c_1 = 2$.  Therefore, the predictive density for $x_{1:2}$ is simply $\{2|x_1-x_2|\}^{-1}$ which is exactly the result given in \citet[][page 309]{berger1998bayes}.  

\citet{berger2001bayesian} argue that, when $p=1$, the conditions of Theorem~\ref{thm:bpv} are satisfied by their P\'olya tree models.  Here we argue that $\{\dpm_{\mu,\sigma}(\alpha,\Psi): (\mu,\sigma) \in \RR^p \times \TT_p\}$ does too, for any $p \geq 1$.  Indeed, it follows immediately from the definition \eqref{eq:dpm} that $\dpm_{\mu,\sigma}(\alpha,\Psi)$ is a location-scale family characterized by the random measure $F^\star \sim \dpm_{0,I}(\alpha,\Psi)$.  Also $F^\star$ is absolutely continuous with respect to the Lebesgue measure because each normal component is so.  Lemma~\ref{lem:rotation} in Appendix~\ref{apndx:proofs} shows that $\dpm_{\mu,\sigma}(\alpha,\Psi)$ is rotation-invariant as well.  Therefore, the result of Theorem~\ref{thm:bpv} holds for the proposed Dirichlet process mixture alternative.  

The null normal model \eqref{eq:null} may be characterized by $\{\langle F_{\mu,\sigma} \rangle: (\mu,\sigma) \in \RR^p \times \TT_p\}$, where $\langle F \rangle$ denotes a degenerate distribution at $F$.  Clearly, this null model is also an absolutely continuous, rotation-invariant, location-scale family, so Theorem~\ref{thm:bpv} applies to the null model too.  Putting these results together leads to the following predictive matching property.  

\begin{theorem}
\label{thm:pred.match}
The two models \eqref{eq:null} and \eqref{eq:alt}, with $\pi_0=\pi_1=\pi_L$, produce the same predictive distribution for any hypothetical sample of size $\nmin = p+1$.  
\end{theorem}

\subsection{Precision parameter and Bayes factor reporting}
\label{SS:alpha}

 With $(\omega_1,\omega_2)$ chosen as in \eqref{eq:omega}, our alternative model and the Bayes factor depend only on the specification of the scalar precision parameter $\alpha$. As discussed in Section \ref{SS:omega}, different values of $\alpha$ allows different amounts and modes of variation of the alternative from the null; see also Figure \ref{fig:omegas}.  Following \citet{berger2001bayesian} we recommend computing the Bayes factor for a range of $\alpha$ values, and presenting them side by side in the form of a plot.  In our examples, we consider a range of $\alpha$ values comparable to that suggested by \citet{escobar1994estimating}. From this plot, the user is free to choose his or her favorite summary of evidence against the null. Various scalar summaries of evidence against $H_0$ can be obtained from this plot. A particularly interesting summary is the minimum Bayes factor. \citet{berger2001bayesian} comment: 
 \begin{quote}
 If this minimum is not small, then there is no reason to doubt $H_0$. Of course, even if this minimum is small, $H_0$ should not be summarily rejected, because the minimum is achieved by searching for the most favorable prior for $H_1$, for the given data, which clearly results in a bias against $H_0$, but, at least, it is useful to know that there are alternatives that better explain the data.
\end{quote}
Some sort of average of the Bayes factors could also be considered. In particular, a weighted harmonic mean of the Bayes factors, weighted according to some probability density $\pi_\alpha$ on $\alpha$, gives the overall Bayes factor with respect to the composite alternative that combines the $\alpha$-indexed family of alternative models through the prior specification $\alpha \sim \pi_\alpha$. However, a non-subjective choice of $\pi_\alpha$ remains an open question \citep[][]{dorazio2009selecting}. Note also that the minimum Bayes factor may be interpreted as the ``empirical Bayes'' Bayes factor because it corresponds to the Type II maximum likelihood estimate of $\alpha$.

\section{Bayes factor computation}
\label{S:computation}

\h{Recall that the Bayes factor $B$ in \eqref{eq:bayes.factor} can be written as the ratio of the marginal likelihood under the null model to that under the alternative.  With $\pi_0(\mu, \sigma) = \pi_L(\mu, \sigma) = \prod_{j=1}^p \sigma_{jj}^{-j}$ one has $\pi_L(\mu, \Sigma) = 2^{-p} \det(\Sigma)^{-(p+1)/2}$. Hence, the marginal data likelihood under the null model, $f_{H_0}(x_{1:n}) = \int \prod_{i=1}^n \nm(x_i \mid \mu, \Sigma) \pi_L(\mu, \Sigma) \, d\mu \, d\Sigma$, equals:
\[
f_{H_0}(x_{1:n}) = \frac{\Gamma_p(\frac{n-1}2)}{2^p n^{p/2} \pi^{p(n-1)/2}\det\{(n-1)S\}^{(n-1)/2}}
\]
where $S$ is the sample variance matrix: $S = (n-1)^{-1} \sum_{i=1}^n (x_i - \bar x)(x_i - \bar x)^\top$.} 
But the denominator of $B$ in \eqref{eq:bayes.factor}, which can be written as  
\begin{equation}
f_{H_1}(x_{1:n}) = \int_{\RR^p \times \TT_p} \int \bigl\{\prod_{i=1}^n dF(x_i) \bigr\} \,d\dpm_{\mu,\sigma}(F \mid \alpha,\Psi) \,d\pi_L(\mu,\sigma),
\label{eq:m1}
\end{equation}
does not yield much analytical simplification and has to be computed by numerical methods. Numerical approximation to marginal likelihoods remains one of the biggest challenges in Bayesian statistics \citep[e.g.,][]{kass1995bayes}, particularly for nonparametric models. For Dirichlet process mixture models, \citet{liu1996nonparametric} presents an efficient sequential imputation algorithm to compute the inner integral in \eqref{eq:m1}. \citet{basu2003marginal} embed this algorithm within the likelihood-posterior ordinate recipe of \citet{chib1995marginal} to approximate \eqref{eq:m1}. We pursue a different adaptation of Liu's algorithm, where we deal with the outer integration in \eqref{eq:m1} by importance sampling and show that it leads to  a quicker and more efficient approximation than the ordinate approach.

\subsection{Importance sampling with sequential imputation}
\label{SS:importance}
Due to the stick-breaking representation \eqref{eq:stick.break}, the alternative model \eqref{eq:alt} on $X_{1:n}$ equals
\begin{equation}
\label{eq:data.conditional}
X_{1:n} \mid \{S_{1:n},(U,V)_{1:n},\mu,\sigma\} \sim \prod_{i=1}^n \nm(X_i \mid \mu + \sigma U_{S_i}, \sigma V_{S_i} \sigma^\top), 
\end{equation}
where $(\mu,\sigma) \sim \pi_L(\mu,\sigma)$, $(U_i, V_i)$, $1 = 1, \ldots, n$, are independent latent mixing parameters drawn from $\Psi$ and $S_{1:n} = (S_1, \ldots, S_n)$ is a vector of labels tracking latent cluster ties. These three sets  of variables are mutually independent. It suffices to restrict the latent labels to the space $\{(s_1, \ldots, s_n) \in \I_n^n : s_1 = 1, s_{i+1} \leq \max(s_{1:i}) + 1, i\in \I_{n-1}\}$, where $\I_n = \{1,\ldots,n\}$. From the P\'olya urn representation \citep{blackwell1973ferguson} of a Dirichlet process, the distribution of $S_{1:n}$ can be written as 
\[ \prob(S_1 = 1) = 1, \quad \prob(S_{i + 1} = \ell \mid S_{1:i}) = 
\begin{cases}
\frac{k_\ell(i) }{\alpha + i} & \ell =1,\ldots,\Lambda_i \\
\frac{\alpha}{\alpha + i} & \ell = \Lambda_i + 1.
\end{cases} \]
where $\Lambda_i = \max(S_{1:i})$ and $K_\ell(i) = |\{j \leq i: S_j = \ell\}|$. 

It is possible to integrate out $U$ from this description, with suitable changes made to \eqref{eq:data.conditional}.  Write $V=V_{1:n}$ and let $f(x_{1:n}, s_{1:n}, v, \mu,\sigma)$ denote the resulting joint density of $(X_{1:n}, S_{1:n}, V, \mu,\sigma)$ and let $f^X_{i+1}(x_{i+1} \mid x_{1:i}, s_{1:i}, v, \mu, \sigma)$ denote the associated conditional density of $X_{i + 1}$ given $(X_{1:i}, S_{1:i}, V, \mu,\sigma)$. Also let $f^S_{i + 1}(s_{i + 1} \mid x_{1:(i + 1)}, s_{1:i}, v, \mu, \sigma)$ denote the conditional density of $S_{i + 1}$ given $(X_{1:(i + 1)}, S_{1:i}, V, \mu, \sigma)$.  These densities are given by 
\begin{align}
f^X_{i + 1}(x_{i + 1} \mid x_{1:i}, s_{1:i}, v, \mu, \sigma) & = \frac{\alpha}{\alpha + i}  \nm(x_{i + 1} \mid \mu, \sigma\sigma^\top) + \sum_{\ell = 1}^{\lambda_i}\frac{ k_\ell(i)}{\alpha + i} \nm(x_{i + 1} \mid \mu_\ell, \sigma_\ell\sigma_\ell^\top), \label{eq:partial.conditional.1} \\
f^S_{i + 1}(\ell \mid x_{1:(i + 1)}, s_{1:i}, v, \mu, \sigma) & = 
\begin{cases}
c^{-1} k_\ell(i)\nm(x_{i + 1} \mid \mu_\ell, \sigma_\ell\sigma_\ell^\top), & \ell =1,\ldots,\lambda_i\\
c^{-1} \alpha \nm(x_{i + 1} \mid \mu, \sigma\sigma^\top), & \ell = \lambda_i + 1,
\end{cases} \label{eq:partial.conditional.2}
\end{align}
with $\lambda_i = \max(s_{1:i})$, $k_\ell(i) = |\{j \leq i: s_j = \ell\}|$, 
\begin{equation}
\label{eq:comp.par}
\begin{split}
\mu_\ell & = \mu + \sigma (I_p - v_\ell) \bigl\{v_\ell + k_\ell(i)(I_p - v_\ell)\bigr\}^{-1} \textstyle\sum_{j = 1}^i (x_j - \mu) 1(s_j = \ell), \\
\sigma_\ell\sigma_\ell^\top & = \sigma v_\ell \bigl\{v_\ell + k_\ell(i)(I_p - v_\ell)\bigr\}^{-1} \bigl\{I_p + k_\ell(i)(I_p - v_\ell)\bigr\}\sigma^\top, 
\end{split}
\end{equation}
 and $c = \alpha \nm(x_{i + 1} \mid \mu, \sigma\sigma^\top) + \sum_{\ell = 1}^{\lambda_i} k_\ell(i)\nm(x_{i + 1} \mid \mu_\ell, \sigma_\ell\sigma_\ell^\top)$.

The marginal likelihood $f_{H_1}(x_{1:n})$ can be calculated by integrating $f(x_{1:n}, s_{1:n}, v,\mu,\sigma)$ with respect to $(s_{1:n}, v,\mu,\sigma)$.  This integral is intractable, but can be approximated by importance sampling Monte Carlo \citep[][Chap.~2.5]{liu2001monte}. Let $(S_{1:n}^m, V^m, \mu^m,\sigma^m)$, $m =1,\ldots,M$, be independent draws from a joint density $f_\imp(s_{1:n}, v,\mu,\sigma)$ on the space of $(S_{1:n}, V,\mu,\sigma)$. Then an unbiased, root-$M$ consistent estimate $f_{H_1}(x_{1:n})$ is 
\h{
\begin{equation}
\label{eq:is.est}
\hat f_{H_1}(x_{1:n}) = \frac1M \sum_{m=1}^M w_m,
\end{equation}
where
\begin{equation}
\label{eq:imp wt}
w_m = \frac{f(x_{1:n}, S_{1:n}^m, V^m, \mu^m,\sigma^m) }{f_\imp(S_{1:n}^m, V^m, \mu^m,\sigma^m)},~~1 \le m \le M,
\end{equation}
are the importance weights of the drawn samples.} The efficiency of this approximation depends on \h{how small the theoretical variance of $w_m$ is, which, in turn, depends on} how well $f_\imp(s_{1:n}, v, \mu,\sigma)$ approximates the conditional density of $(S_{1:n}, V,\mu,\sigma)$, given $X_{1:n} = x_{1:n}$, under the joint density $f(x_{1:n}, s_{1:n}, v,\mu,\sigma)$; refer to \citet{tokdar2010importance} for an overview importance sampling theory.  Below we present one choice that gives a good approximation.

Let $f_\imp(s_{1:n}, v, \mu, \sigma)$ be the joint density of $(S_{1:n}, V,\mu,\sigma)$ where $(\mu,\sigma)$ has density $f^{\mu,\sigma}_\imp(\mu,\sigma)$ to be specified later, $V=(V_1, \ldots, V_n)$ are independent draws from $\be(\omega_1, \omega_2)$, also drawn independently of $(\mu,\sigma)$, and $S_{1:n}$ given $V = v$ and $(\mu,\sigma)$ has density $\prod_{i = 0}^{n - 1} f_{i + 1}^S(s_{i + 1} \mid x_{1:(i+1)}, s_{1:i}, v, \mu, \sigma)$ as given in \eqref{eq:partial.conditional.2}. This choice can be justified on two accounts.  First, the conditional importance density of $S_{i + 1}$   given $(S_{1:i}, V, \mu,\sigma)$ is the partial conditional density of $S_{i + 1}$ given $(X_{1:(i+1)}, S_{1:i}, V, \mu, \sigma)$ under $f$.  Second, the partial conditional density under $f$ of $V_\ell$ given $\{S_{i + 1} = \max(S_{1:i}) + 1 = \ell, X_{1:(i+1)}, V_{1:(\ell - 1)}, \mu,\sigma\}$ is $\be(\omega_1, \omega_2)$.  Using the sequential imputation calculations of \citet{liu1996nonparametric} and the definition of $f_\imp$, it can be shown that 
\[ f(x_{1:n}, s_{1:n}, v,\mu,\sigma) = f_\imp(s_{1:n}, v, \mu,\sigma)\frac{\pi_L(\mu,\sigma)}{f_\imp^{\mu,\sigma}(\mu,\sigma)} \prod_{i = 0}^{n - 1} f^X_{i + 1}(x_{i + 1} \mid x_{1:i}, s_{1:i}, v, \mu, \sigma). \]
Therefore for each $m \in \{1, \ldots, M\}$, the corresponding importance weight in \eqref{eq:imp wt} can be expressed as
\begin{align}
w_m  & = \frac{\pi_L(\mu^m,\sigma^m)}{f_\imp^{\mu,\sigma}(\mu^m,\sigma^m)} \prod_{i = 0}^{n - 1} f^X_{i + 1}(x_{i + 1} | x_{1:i}, s_{1:i}^m, v^m, \mu^m, \sigma^m) \nonumber\\
& =  \frac{\pi_L(\mu^m,\sigma^m)}{f_\imp^{\mu,\sigma}(\mu^m,\sigma^m)} \prod_{i = 0}^{n - 1} \frac{\alpha\nm(x_{i + 1} \mid \mu^{ m }, \sigma^{ m }\sigma^{ m\top }) +  \sum_{\ell = 1}^{\lambda_i^m} k_\ell^m(i) \nm(x_{i + 1} \mid \mu^{ m }_\ell, \sigma_\ell^{ m }\sigma^{ m\top }_\ell)}{\alpha + i}, \nonumber
\end{align}
with formulas for $\mu^m_\ell$ and $\sigma^m_\ell$ suitably adapted from \eqref{eq:comp.par}; similarly for $\lambda_i^m$ and $k_\ell^m(i)$. \h{In order to obtain more easily programmable formulas of this expression we recommend: (A) standardizing the observations $x_{1:n}$ with respect to mean $\mu^m$ and variance $\sigma^m{\sigma^{m\top}}$, and (B) using a spectral decomposition of $v^m_i = Q^m_i D^m_i {Q^m_i}'$, $i = 1,\ldots, n$.}

\subsection{Choice of importance density on $(\mu, \sigma)$}

In order to make the importance sampling estimate efficient, $f_\imp^{\mu,\sigma}(\mu,\sigma)$ should be chosen to approximate $f_{H_1}(\mu,\sigma \mid x_{1:n})$, the posterior density of $(\mu,\sigma)$ under $H_1$. Due to the embedding and predictive matching properties of the alternative, one may expect the posterior density of $(\mu,\sigma)$ under the alternative to be similar to that under the null. A reasonable and conservative choice is an approximation to $f_{H_0}(\mu,\sigma \mid x_{1:n})$ with heavier tails to guard against a possible mismatch with  $f_{H_1}(\mu,\sigma \mid x_{1:n})$ \citep{berger2001bayesian}.

For any $\nu > p-1$ and any $\Psi \in \SS_p$, let $\wish(\nu, \Psi)$ and $\iwish(\nu, \Psi)$ denote, respectively, the Wishart and the inverse-Wishart distributions with shape $\nu$ and scale $\Psi$. The null posterior density, viewed through the $(\mu, \Sigma)$ parametrization, could be conveniently written as: 
\[
f_{H_0}(\mu, \Sigma \mid x_{1:n}) = \iwish(\Sigma \mid n-1, (n-1)S) \times \nm(\mu \mid \bar x, n^{-1}\Sigma).
\]
We take $f^{\mu,\sigma}_\imp(\mu, \sigma)$ to be a heavy tailed approximation to it, controlled by two scalar parameters $\nu > p-1$, $\rho > 0$, and given by:
\[
f^{\mu,\Sigma}_{\imp}(\mu, \Sigma) = \fdist(\Sigma \mid \nu - (p-1), \nu, S) \times \stt_\nu(\mu \mid \bar x, \rho \Sigma / n)
\]
where $\fdist(\kappa, \delta, \Psi)$ denotes the matrix F distribution \citep{mulder2018matrix} with shapes $\kappa > 0$, $\delta > p-1$ and scale $\Psi \in \SS_p$. Under this importance density one could write $\Sigma \mid \Phi \sim \iwish(\nu, \Phi)$, $\Phi \sim \wish(\nu, S)$. See Appendix~\ref{apndx:matF} for more details on efficient random sampling and probability density evaluation of such a $\Sigma$. In our numerical experiments, we used $\nu = \max\{p+1, n - p\sqrt n\}$ and $\rho = \sqrt{n}$; but the results were not too sensitive to these choices.

\subsection{Comparison with \citet{basu2003marginal}}
\label{SS:basu-chib}
The likelihood-posterior ordinate recipe of \citet{chib1995marginal} approximates $f_{H_1}(x_{1:n})$ by the quantity $\pi_L(\mu^\star, \sigma^\star) f_{H_1}(x_{1:n} |  \mu^\star, \sigma^\star)/ f_{H_1}(\mu^\star, \sigma^\star | x_{1:n})$ where $(\mu^\star, \sigma^\star)$ is any point of high posterior density. To approximate the likelihood ordinate $f_{H_1}(x_{1:n} | \mu^\star, \sigma^\star)$ once a $(\mu^\star, \sigma^\star)$ has been chosen, \citet{basu2003marginal} recommend using the importance sampling scheme on $(S_{1:n}, V)$ described in Section \ref{SS:importance} conditional on $\mu = \mu^\star, \sigma = \sigma^\star$, leading to the following approximation to $B$
\begin{equation}
\widehat B^{-1} = \frac{f_{H_0}(\mu^\star,\sigma^\star \mid x_{1:n})}{f_{H_1}(\mu^\star,\sigma^\star \mid x_{1:n})}\frac{1}{M} \sum_{m = 1}^M  \prod_{i = 0}^{n - 1} \left\{ \frac\alpha{\alpha + i} +  \sum_{\ell = 1}^{\lambda_i^m} \frac{k_\ell^m(i)}{\alpha + i} \frac{\nm(x_{i + 1} \mid \mu^{ m }_\ell, \sigma_\ell^{ m }\sigma^{ m\top }_\ell)}{\nm(x_{i + 1} \mid \mu^{ \star }, \sigma^{ \star }\sigma^{ \star\top })}\right\}
\label{eq:plo}
\end{equation}
\citet{basu2003marginal} recommend identifying $(\mu^\star, \sigma^\star)$ by running an initial Markov chain sampler, preferably a Gibbs sampler which can also provide a Rao--Blackwellized Monte Carlo approximation to the posterior ordinate $f_{H_1}(\mu^\star, \sigma^\star | x_{1:n})$. Alternatively one could gather posterior samples of $(\mu, \sigma)$ and use efficient smoothing based density estimation techniques to approximate $f_{H_1}(\mu^\star, \sigma^\star | x_{1:n})$. We follow both suggestions to construct two competitors of our importance sampling algorithm for the univariate case.

\begin{table}
\centering
\begin{tabular}{p{1.5in}rrrrrrr}
Algorithm & Mean & Min. & 1st Q. & Median & 3rd Q. & Max& Time\\
\hline
Basu--Chib & 1.66 & 0.0005 & 0.86 & 1.61 & 2.32 & 7.7 & 3.6s\\
Basu--Chib + smoothing & 1.82  & 1.39 & 1.68  & 1.79  & 1.92  & 2.95  & 3.7s\\
Importance sampling & 1.82 & 1.54 & 1.76 & 1.81 & 1.87 & 2.03 & 0.8s\\[5pt]
\end{tabular}
\caption{Comparison of our importance sampling method against two versions of the \citet{basu2003marginal} algorithm; see text for more details. Columns 2 through 7 give summaries of 100 replications of the Bayes factor computation on the same data set of 100 draws from the standard normal density. Last column refers to run time in seconds per computation.}
\label{tab:basu-chib}
\end{table}

Table \ref{tab:basu-chib} gives summaries of 100 replications of the Bayes factor computation on a single synthetic data set we simulated with 100 draws from the standard normal density. ``Basu--Chib'' refers to Monte Carlo posterior ordinate approximation based on a Gibbs sampler, which is fairly straightforward to design for our choice of Dirichlet process mixture \citep[see e.g.,][for a basic construction]{escobar1995bayesian}. ``Basu--Chib + smoothing'' refers to posterior ordinate approximation based on kernel smoothing of the Gibbs sampler draws of $(\mu, \sigma)$. Smoothing was done by the \texttt{kde} function of the \texttt{R}-package \texttt{ks}, with bandwidth chosen by the plug-in method of \citet{wand1994multivariate}. We also tried the more computationally expensive cross-validation choice of the bandwidth \citep{duong2005cross} which did not result in any appreciable improvement in performance (not reported). ``Importance sampling'' refers to our approach. Each algorithm was run with 10,000 importance samples. ``Basu--Chib'' algorithm required two additional runs of the Gibbs sampler, one to identify $\mu^\star$, $\sigma^\star$ as median draws and the other to approximate the posterior ordinate. ``Basu--Chib + smoothing'' requires only one run of the Gibbs sampler to simultaneously identify $\mu^\star, \sigma^\star$ and gather posterior draws of $\mu, \sigma$ to be used in smoothing. All runs of Gibbs sampler were 10,000 iterations each. 

Table \ref{tab:basu-chib} makes it clear that our importance sampling approach offers a more efficient estimation of the Bayes factor with substantially lower computing cost than either Basu--Chib algorithm. The posterior ordinate approximation step appears suspect for the poor performance of the latter. Smoothing helps, but not to the extent to make the likelihood-posterior ordinate method competitive against our importance sampling algorithm.

\subsection{Additional considerations for multivariate data}
\h{A weakness of the importance sampling scheme described above is that the sampling of the atoms $V_{1:n}$ does not incorporate any information from the data. Each time an observation is assigned to a new cluster $\ell$, the cluster's variance component $V_\ell$ is sampled from the prior and never updated. This could be particularly troublesome in higher dimensions where the chances of randomly landing on an appropriate $V_\ell$ for each new cluster are very slim. As a possible mitigation of this sampling inefficiency, we propose a Rao-Blackwellization extension inspired by sequential Monte Carlo ideas where each $V_\ell$ is represented by a set of particles $V^\star_{\ell r}$, $r = 1, \ldots,R$, whose weights are updated every time a new observation is added to the $\ell$-th cluster.}

To be more precise, notice that the sampling of $V_{1:n}$ under the alternative prior could be represented as: for each $i=1,\ldots,n$, sample $V^\star_{ir} \sim \be(\omega_1, \omega_2)$, $r = 1, \ldots, R$, independently of each other, and then set $V_i$ to be one of the $V^\star_{ir}$ chosen at random. With $V^\star = V^\star_{1:n, 1:R}$, one can then integrate $V_{1:n}$ from the model and rewrite \eqref{eq:partial.conditional.1} and \eqref{eq:partial.conditional.2} as
\begin{align}
f^X_{i + 1}(x_{i + 1} \mid x_{1:i}, s_{1:i}, v^\star, \mu, \sigma) & = \frac{\alpha}{\alpha + i} \nm(x_{i + 1} \mid \mu, \sigma\sigma^\top)\nonumber\\
& \quad\quad\quad\quad + \sum_{\ell = 1}^{\lambda_i}\frac{ k_\ell(i)}{\alpha + i}\sum_{r=1}^R q_{\ell r}(i) \nm(x_{i + 1} \mid \mu_{\ell r}, \sigma_{\ell r}\sigma_{\ell r}^\top), \label{eq:partial.conditional.1R} \\
f^S_{i + 1}(\ell \mid x_{1:(i + 1)}, s_{1:i}, v^\star, \mu, \sigma) & = 
\begin{cases}
c^{-1} k_\ell(i) \sum_{r = 1}^R q_{\ell r}(i)\nm(x_{i + 1} \mid \mu_{\ell r}, \sigma_{\ell r}\sigma_{\ell r}^\top), & \ell =1,\ldots,\lambda_i\\
c^{-1} \alpha \nm(x_{i + 1} \mid \mu, \sigma\sigma^\top), & \ell = \lambda_i + 1,
\end{cases} \label{eq:partial.conditional.2R}
\end{align}
where $\mu_{\ell r}$ and $\sigma_{\ell r}$ are computed as in \eqref{eq:comp.par} but with $v^\star_{\ell r}$ instead of $v_\ell$, and, $q_{\ell r}(i) = m_{\ell r}(i)/ \{m_{\ell, 1}(i) + \cdots + m_{\ell, R}(i)\}$, $r = 1, \ldots, R$, with
\[
m_{\ell r}(i) =\frac{\exp[-k_\ell(i) {\rm tr}\{S_{\ell}(i) (\sigma v^\star_{\ell r}\sigma^\top)^{-1}\}]}{ \det(v^\star_{\ell r})^{\frac{k_\ell(i)-1}{2}} } \nm\left(\bar x_\ell(i) \mid \mu, \sigma \left\{{v^\star_{\ell r}}/{k_\ell(i)} + I_p - v^\star_{\ell r}\right\}\sigma^\top \right),
\]
where $\bar x_\ell(i) = k_\ell(i)^{-1}\sum_{j \le i} x_j I(s_j = \ell)$ is the current cluster mean and $S_\ell(i) = k_\ell(i)^{-1} \sum_{j \le i} (x_j - \bar x_\ell(i))(x_j - \bar x_\ell(i))^\top I(s_j = \ell)$ is the current cluster variance when only the first $i$ observations have been processed.

\h{
The corresponding importance sampling density for $(S_{1:n}, V^\star, \mu, \sigma)$ is defined analogously with $V$ replaced by $V^\star$. The calculation of the importance weights is modified accordingly. Maintaining and updating the relative weights of the $R$ particles for each cluster parameter $V_\ell$ offers a greater incorporation of the observed data. We expect a larger $R$ to be needed for higher dimension, as the space of $V$ matrices is $p(p+1)/2$ dimensional. We suggest a default choice of $R = p(p+1)$, although a thorough investigation of this choice is beyond the scope of this paper. The numerical experiments reported in the next section were carried out with this default choice.
}

\section{Case studies}
\label{S:illustrations}

\begin{example}
\label{ex:kevlar}
\citet{berger2001bayesian} illustrate their Polya tree test on the log-lifetime measurements of 100 Kevlar pressure vessels \citep[][p.~183]{andrews1985data}.  Our alternative model produces a minimum Bayes factor close to $10^{-5}$ for $\alpha \in [2^{-6},2^{13}]$, showing negligible evidence toward normality. We used 20,000 importance samples to compute $\hat B$. Our minimum Bayes factor is similar in magnitude to the one reported by \citet{berger2001bayesian}.  
\end{example}

\begin{example}
\label{ex:multivariate.sims}
Figure~\ref{fig:mv.sims} shows scatterplots of three synthetic datasets of size $n = 100$ and dimension $p = 2$ simulated respectively from a bivariate standard normal, a bivariate standard Student-t with three degrees of freedom, \h{and a Frank copula distribution (parameter = 20 which corresponds to Kendall's $\tau=0.816$) with standard normal marginals. For each dataset, the Bayes factor was calculated over a regular grid of $\alpha \in [2^{-6}, 2^{13}]$ with 10,000 importance samples. Each calculation was replicated 8 independent times to assess reproducibility. We report in Figure \ref{fig:mv.sims} the median, minimum and maximum of these 8 log Bayes factor evaluations, along with a combined estimate obtained by pooling the $8\times10,000$ importance samples together. The graphs indeed suggest that these calculations were fairly reproducible.}

\begin{figure}[!t]
\begin{center}
\scalebox{0.4}{\includegraphics{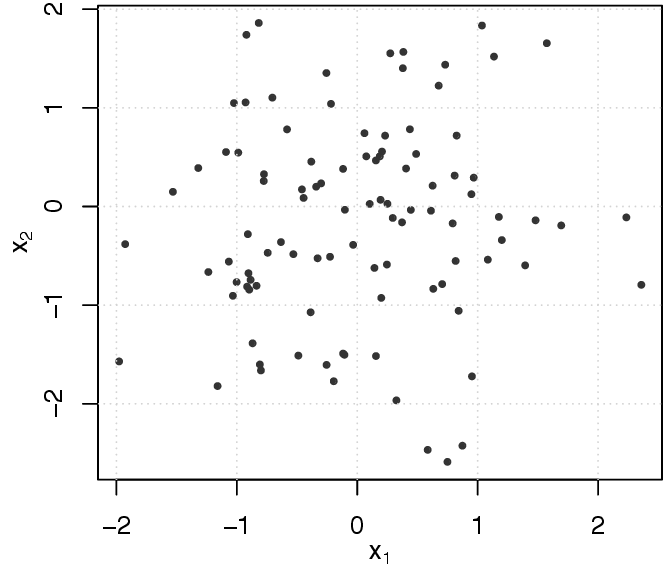}}
\scalebox{0.4}{\includegraphics{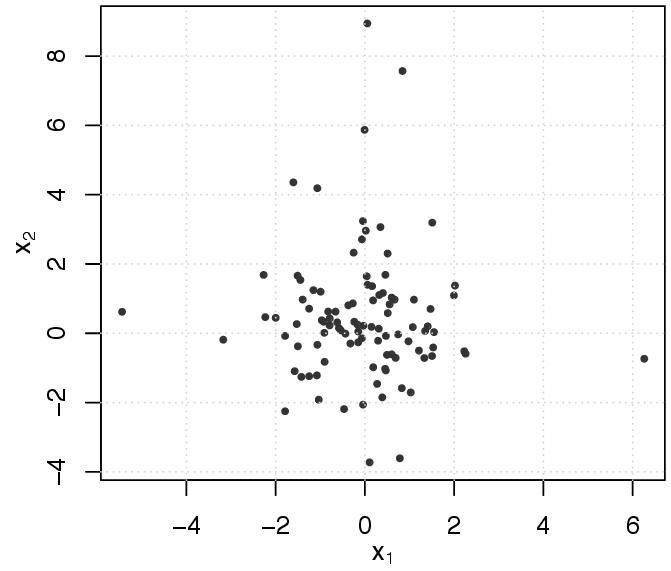}}
\scalebox{0.4}{\includegraphics{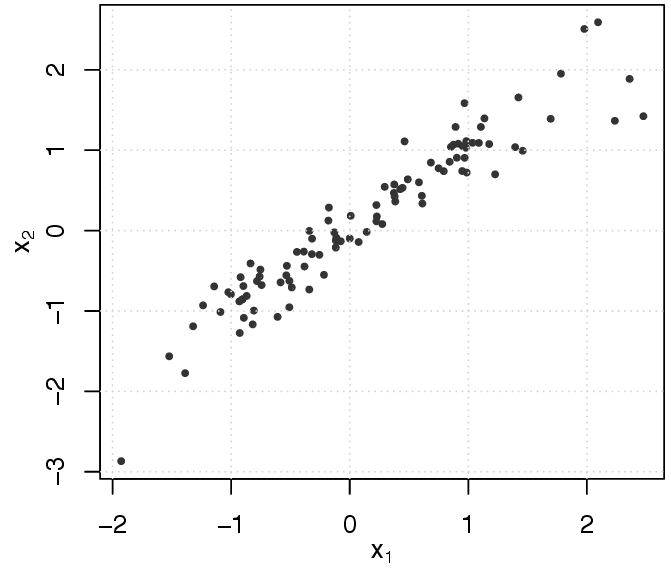}}\\
\scalebox{0.4}{\includegraphics{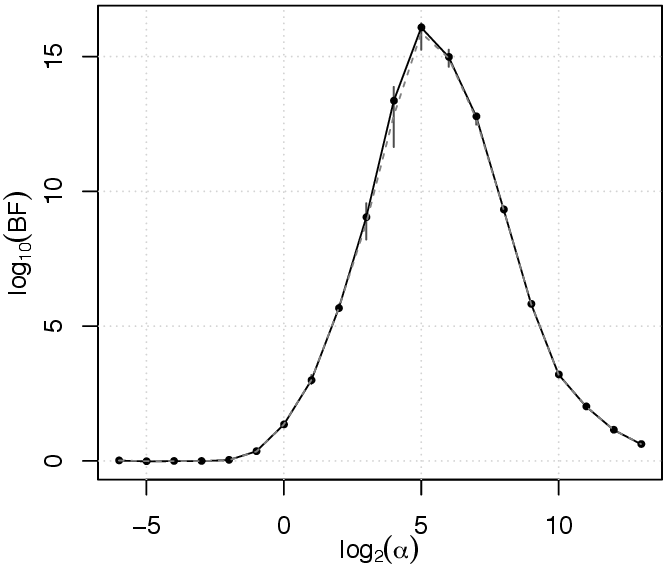}}
\scalebox{0.4}{\includegraphics{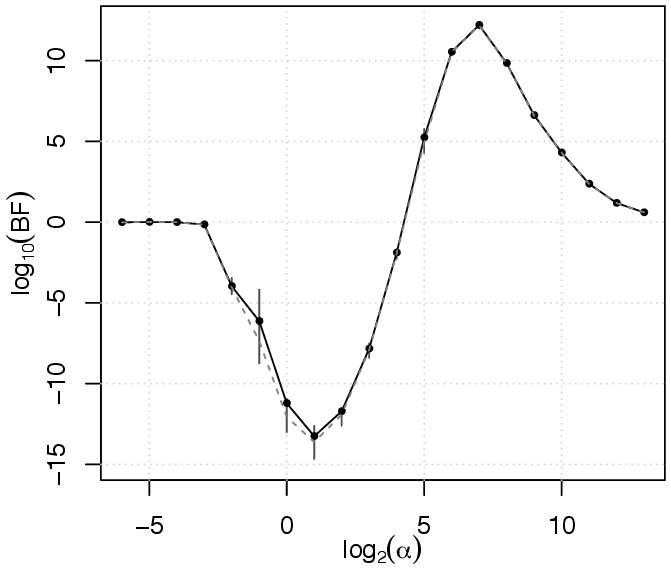}}
\scalebox{0.4}{\includegraphics{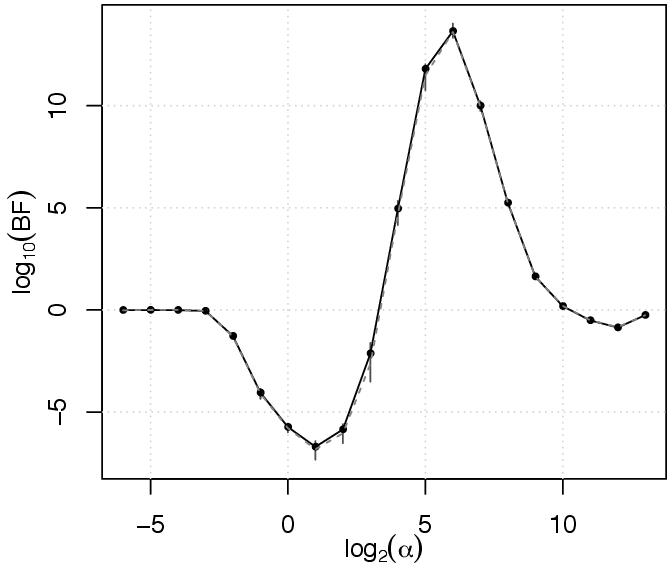}}
\end{center}
\caption{\h{Bayes factors (top) for data (bottom) in Example~\ref{ex:multivariate.sims}; normal (left), Student-t (middle), Frank copula with Gaussian marginals (right). For each dataset, Bayes factors were calculated for precision parameter $\alpha =2^{k}$ with an integer $k$ running from $-6$ through 13. Each Bayes factor calculation was repeated 8 times with 10,000 importance samples each. The median, minimum and maximum Bayes factor values (in base 10 logarithm) are shown as a vertical segment superimposed with a filled circle. The dashed line shows a combined estimate of the Bayes factor by pooling together all the $8\times10,000$ importance sample draws.}}
\label{fig:mv.sims}
\end{figure}

\h{For the first dataset the Bayes factor remains essentially larger than 1 for all precision parameter values} and becomes quite large for moderate $\alpha$, indicating little doubt against normality. The Student-t dataset shows a concentration of points around $(0,0)$ along with a number of outliers, suggesting a heavier-than-normal tail.  The Bayes factor bottoms out \h{around $10^{-12}$} with fairly small values in the range $0.5 \le \alpha \le 8$ where the prior encourages a moderate number of clusters. Presumably, the alternative is able to pick up the heavy tail by assigning the large outliers into separate clusters. For the copula dataset with a non-elliptical scatter, the Bayes factor \h{achieves a minimum of $10^{-7}$} suggesting strong evidence against normality. Interestingly, the Bayes factor dips below 1 across two distinct segments of $\alpha$ values: $\alpha \le 8$ and $\alpha \ge 2^{11}$. The latter segment corresponds to, {\it a priori}, a large number of mixture components where each component has a tiny volume share (refer to Figure \ref{fig:omegas}).  
\end{example}

\begin{figure}[!tp]
\begin{center}
\subfigure[Data scatter and histograms]{\scalebox{0.6}{\includegraphics{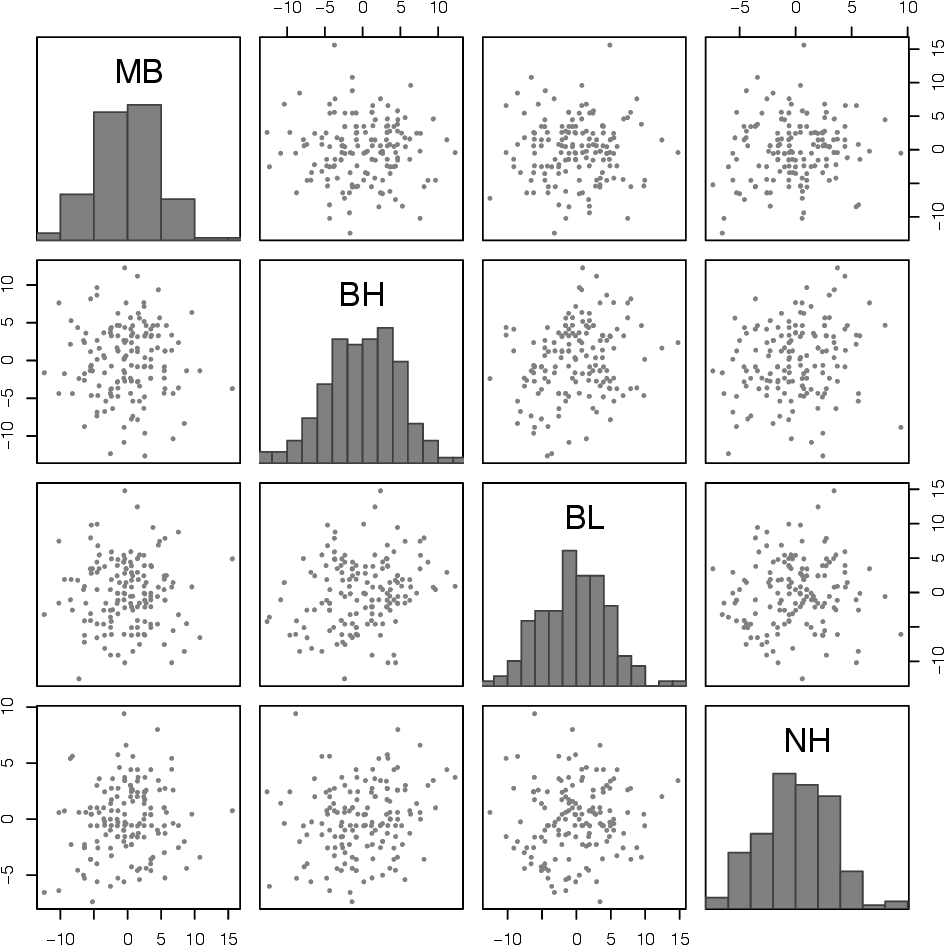}}}
\subfigure[Individual tests of normality]{\scalebox{0.4}{\includegraphics{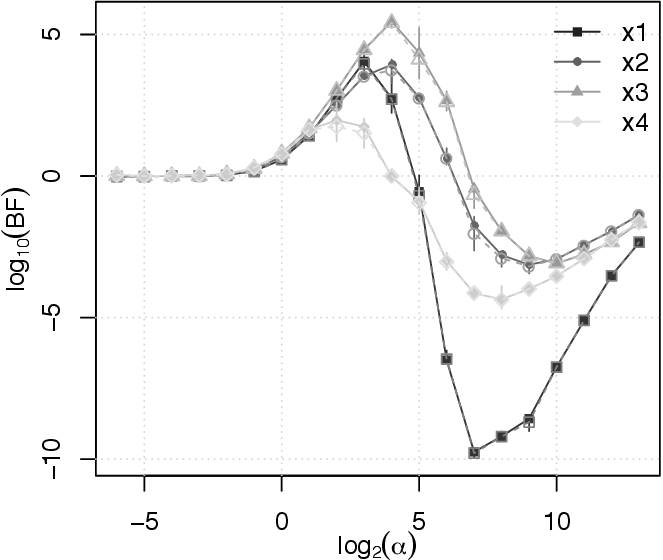}}}
\subfigure[Joint test of normality]{\scalebox{0.4}{\includegraphics{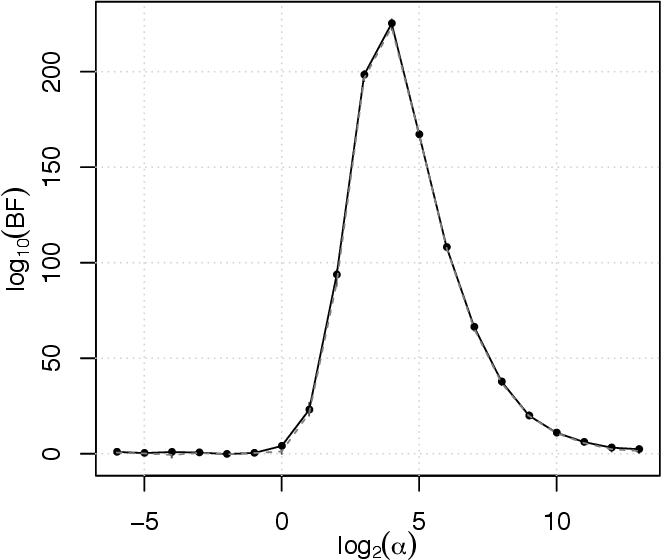}}}
\end{center}
\caption{Egyptian skull data analysis. \h{For the single variable analyses, each Bayes factor evaluation was performed using 10,000 importance samples and repeated 8 times. Vertical segments and filled shapes show the minimum, maximum and the median of the log Bayes factor values. A combined evaluation of the same, obtained by pooling all 80,000 samples together, is reported via the dashed line. Same evaluation strategy was adopted for the joint analysis, but here each evaluation was done using 50,000 importance samples.}}
\label{fig:skull}
\end{figure}

\begin{example}
\label{ex:skull}

The well known Egyptian skulls dataset \citep{hand1994handbook} consists of $p=4$ measurements \h{(MB: maximal breadth, BH: basibregmatic height, BL: basialveolar length, and, NH: nasal height, all in mm)} 
taken on $n=150$ ancient Egyptian skulls from five time epochs during 4000 B.C.~and 200 A.D. Mean effect of time was removed by running a multivariate analysis of variance, and we test the residuals for normality. \h{Ties were broken by injecting a small jitter to each observation with random uniform draws between $[-1/60, 1/60]$}. A chi-square QQ-plot (not shown) of the Mahalanobis distance squares showed a faint deviation from normality. Mardia's skewness and kurtosis tests failed to reject normality with p-values $\approx 0.5$. 

\h{
Interestingly, our approach showed moderately strong evidence against normality when the four measurements were analyzed separately, but little evidence against normality for their joint distribution (Figure \ref{fig:skull}). This situation is very different from the copula example considered above, where marginal distributions were normal but non-normality could be detected for the joint distribution. A potential explanation is that given the predictive matching property of our approach, it will take more observations in 4 dimensions to detect non-normality than in a univariate case. For example, if we had only 5 observations from a 4 dimensional distribution, our approach is guaranteed to produce a Bayes factor of 1 in the joint analysis, but may be able to detect non-normality of the marginals. However, this does not fully explain the stark difference between the univariate and the multivariate results reported in Figure \ref{fig:skull}. It is possible that the joint distribution deviates from normality in ways that are not well captured by the mixture alternative proposed here, whereas univariate projections are well approximated as mixtures of normals. 
}

\end{example}

\section{Numerical experiments}
\label{S:simu}

\subsection{Power-size comparison against Polya tree and Anderson--Darling}
\label{SS:comp}

Comparing the minimum Bayes factor against a threshold gives a goodness-of-fit test of normality in the classical sense, subject to size and power calculations. Size may be approximated by simulating data from the null. Due to Lemma \ref{lem:inv}, it is sufficient to simulate under any one normal distribution because of the location-scale invariance nature of our alternative specification. We ran a simulation study to compare size and power of the resulting tests to tests derived similarly from the Polya tree approach of \citet{berger2001bayesian} and the classical Anderson-Darling tests. For our approach, minimum Bayes factor was calculated over $\alpha \in [2^{-6},2^4]$. For the Polya tree tests, we used the fixed-partition (Type~2) version \citep[][Equation~2]{berger2001bayesian} with the function $d(\eps_m) = h^{-1} 4^m$ for scale parameter and calculated minimum Bayes factor over $h \in [2^{-6},2^4]$.

Size calculations were done with 100 datasets each consisting of $n = 100$ draws from the standard normal distribution. For power calculation under the alternative, we considered three non-normal distributions: Student-t with 3 degrees of freedom, skew-normal with shape parameter 10, and uniform on the interval $(-1,1)$. For any of these three distributions, power was approximated by simulating 100 datasets each with $n = 100$ draws from the distribution. Results are shown in Figure~\ref{fig:roc} as power-size curves for three sets of tests for each of the chosen non-normal distributions.  For all three distributions, the Dirichlet process mixture tests perform the best, producing higher power at a lower smaller size. Anderson--Darling tests generally outperform the Polya tree tests.  The results for the uniform distribution were surprising to us as we expected the Polya-tree alternative to beat Dirichlet process mixtures at detecting discontinuities. 

\begin{figure}
\begin{center}
\subfigure[Student-t, degrees of freedom 3]{\scalebox{0.4}{\includegraphics{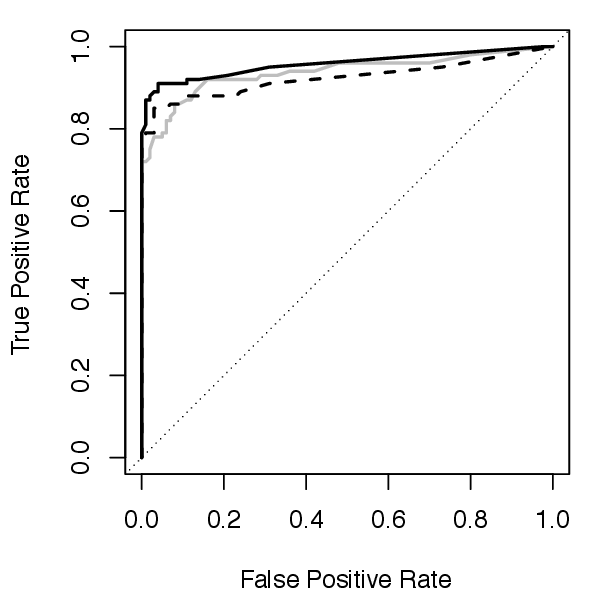}}}
\subfigure[Skew-normal, shape = 10]{\scalebox{0.4}{\includegraphics{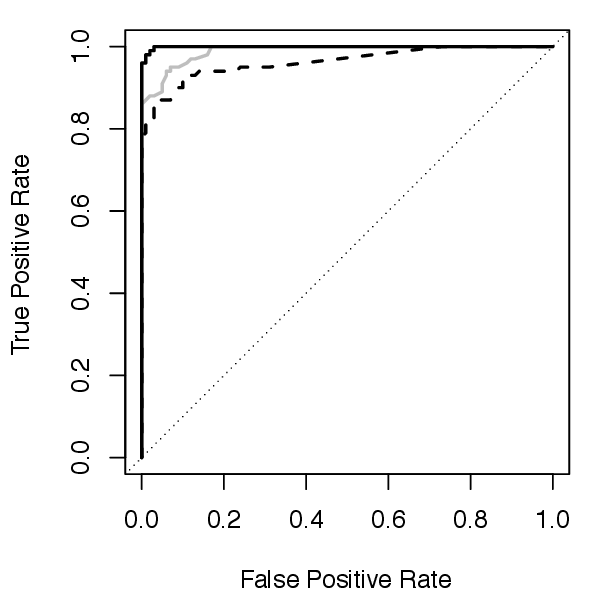}}}
\subfigure[Uniform on $(-1,1)$]{\scalebox{0.4}{\includegraphics{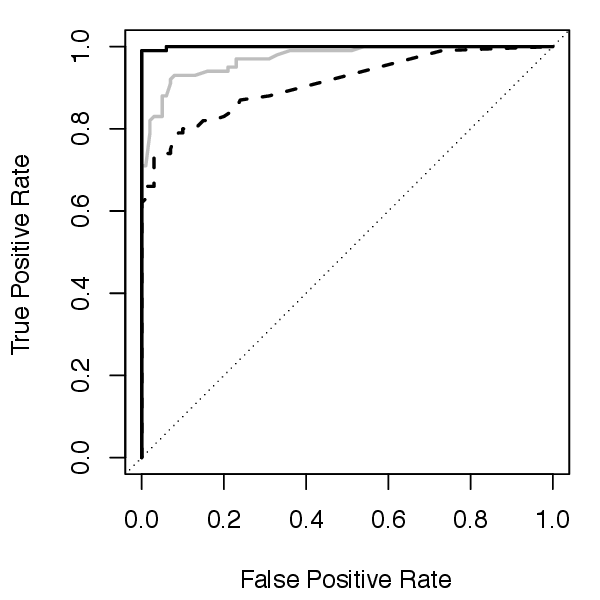}}}
\end{center}
\caption{Power-size curves for Dirichlet process mixture (solid, black), Polya tree (broken, black) and Anderson--Darling (solid, gray) tests. Each panel presents power for a specific alternative benchmarked against size under the null hypothesis of Gaussianity.}
\label{fig:roc}
\end{figure}

\subsection{Bayes factor consistency}
\label{SS:consistency}

A desirable frequentist property of a Bayesian testing procedure is Bayes factor consistency, i.e., the Bayes factor should converge to $\infty$ asymptotically under the null, and to $0$ under the alternative as sample size grows to infinity. It follows from a simple argument \citep[e.g.][Section 4]{tokdar2010bayesian} that $B \to \infty$ almost surely whenever $X_i$'s are drawn from a non-normal distribution that is in the Kullback--Leibler support of the alternative prior distribution \citep{ghosh2003bayesian}. Substantial existing literature \citep{ghosal1999posterior, tokdar2006posterior, ghosal2007posterior, shen2013adaptive} indicates that Dirichlet process mixtures of normals prior distributions have broad Kullback--Leibler support which can be characterized by mild continuity and tail conditions. The same could be expected for our non-parametric prior, although formal details will be different. Proving $B \to \infty$ under the null is much more challenging and requires showing the non-parametric prior is less densely packed around any normal distribution than what a parametric prior will be \citep[][Section 4]{tokdar2010bayesian}. Such lower bounds on prior concentration and technical tools needed to prove them are scarce in the literature and have only been established formally for relatively simple kernel mixtures \citep{mcvinish2009bayesian}.

We ran a simulation study to assess Bayes factory consistency under the null. We simulated 100 independent standard normal data sequences of length 5000, and evaluated the Bayes factor (with fixed $\alpha=1$) at several points $n$ along each of the sequences.  The Bayes factor paths are displayed in Figure~\ref{fig:bf.limit} along with the 2.5\%, 50\%, and 97.5\% quantile paths.  The Bayes factor sampling distribution appears to be shifting upwards with $n$. Moreover, $\prob(B > 1)$ seems to converge to 1 with $n$ and appears to be at least 0.975 for $n \geq 5000$. Although this experiment does not cover all interesting scenarios, it gives substantial evidence that $B \to \infty$ in probability, under the null, as $n \to \infty$.  

\begin{figure}
\begin{center}
\scalebox{0.5}{\includegraphics{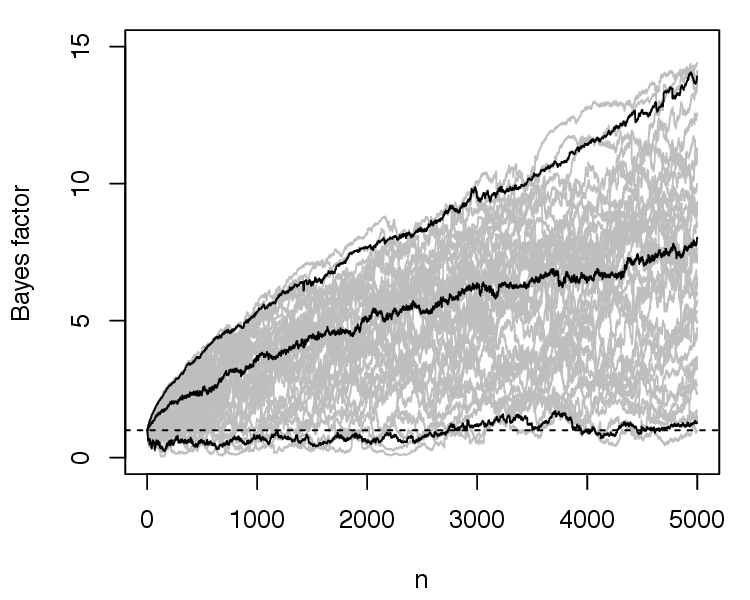}}
\end{center}
\caption{Bayes factor sample paths, with 2.5\%, 50\%, and 97.5\% summaries, when 100 datasets were generated from a standard normal distribution. Each path represents one dataset, depicting as a function of $n \in \mathbb{N}$, the Bayes factor values calculated based on the first $n$ observations in the dataset. About half of the sample paths are shown to improve clarity.}
\label{fig:bf.limit}
\end{figure}

\h{

\section{Concluding remarks}
\label{S:discuss}

We have presented here a novel Bayesian assessment of normality for univariate or multivariate data under the formal guidelines of null embedding and predictive matching advocated by \citet{berger2001bayesian}. A broad nonparametric alternative to normality is proposed based on a Dirichlet process mixture of normals. We show that the alternative space partitions into disjoint sets each of which can be mapped to a single Gaussian density identified by its location and scale. Each partition consists of densities that are clustering based granulations of the corresponding Gaussian. 
Our specification relies on a new type of Dirichlet process mixture of normals that enables such an embedding, generalizing the constructions of \citet{griffin2010default} to higher dimensions. 

A key theoretical contribution lies in establishing a predictive matching property of this new alternative class when the left Haar prior is used on the location and scale parameters for both the null and alternative hypotheses. Consequently, the new test remains deliberately neutral between the null and the alternative until at least $n = p+1$ samples are available, which is the minimum sample size needed to estimate the location and scale parameters of a $p$ dimensional normal distribution.

A sequential importance sampling Monte Carlo estimate is proposed toward reasonably fast and reproducible evaluations of Bayes factors. Our development utilizes a Rao--Blackwellized extension of the sequential imputation technique of \citet{liu1996nonparametric} where each atom of the Dirichlet process distribution is represented by a set of particles whose weights are updated every time a new sample is associated with that atom. An R package (\texttt{gausstest}) is currently under development. A preliminary version is available at \url{https://github.com/tokdarstat/gausstest}.

We have presented simulation studies demonstrating that the proposed method has higher discriminatory power when the true, data-generating distribution is a smooth departure from normality, and also avoids over-fitting when the true distribution is normal. We have also presented numerical evidence that the resulting Bayes factor is likely to be asymptotically consistent in discriminating between Gaussian and non-Gaussian distributions. However, much work remains to be done to establish this rigorously. 
  
Our analysis of the Egyptian skull data opens up new questions. For this dataset, our new test detected non-Gaussianity individually for each of the four measurements, but found little evidence against Gaussianity for their joint distribution. Although our test requires more data to detect non-Gaussianity in higher dimensions, this fact alone does not fully explain the contrasting results we reported for this analysis. It is not clear under what simulation models, if any, similar phenomena could be observed. It is also not clear what insights practitioners would draw when such contrasting results manifest.

Our approach does involve a single tuning parameter: the precision parameter $\alpha$ of the Dirichlet distribution. Our numerical studies show that the Bayes factor is quite sensitive to this parameter value. Following \cite{berger2001bayesian}, we have adopted the approach of evaluating the Bayes factor across a wide range of precision parameter values, and reporting the minimum Bayes factor as the (worst case) evidence against Gaussianity. We believe this to be reasonable as each distinct value of $\alpha$ corresponds to a distinct alternative subspace (Figure \ref{fig:omegas}), and a goodness of fit test should try to gather maximum evidence against the null model within the full space of plausible alternatives. 

A related issue is that we have opted for a very specific relationship between $\alpha$ and the shape parameters $\omega_1$ and $\omega_2$, which control the relative spread of the clusters with respect to the total data spread. So, $\alpha$ not only controls the number of clusters, it also controls the relative volume shares of the clusters. We have adopted a very specific relationship that achieves certain limiting properties, but our choice is rather ad-hoc. A more systematic study is needed to understand this issue better.

Despite the promise of our computational algorithm, scaling it up to large number of observations or high dimensional data remains a formidable challenge. Importance sampling may not be the right approach for even moderately high dimensional observations (say $p > 5$). Sequential imputation may not be very efficient when $n$ is reasonably large (say $n > 1000$). It will be interesting to see if sequential Monte Carlo techniques e.g., \citet{griffin2017sequential} could yield more efficient computing algorithms. 

The goodness of fit test proposed here is exclusive to detection of non-Gaussianity. It is not immediately clear what other kinds of parametric models could be assessed with a similar approach. Our construction of the local alternatives utilizes the fact that a convolution of Gaussians is a Gaussian itself. The same holds for all infinitely divisive distributions. Conceivably, similar constructions could be done for testing the fit of a specific infinitely divisible, location-scale family model. 
}

\section*{Acknowledgment}

We thank the Editor and two reviewers whose comments on an earlier draft led to considerable improvement of the article.

\section*{Appendix}

\appendix
\section{Proofs}
\label{apndx:proofs}

\begin{proof}[Proof of Theorem~\ref{thm:neg.cov}]
Since $(U \mid V) \sim \nm(0, I_p-V)$, it follows that $\E(U^\top U \mid V) = \tr\E(UU^\top \mid V) = \tr(I_p-V) = p - \tr V$, where $\tr A$ returns the trace of a symmetric matrix $A$.  Then 
\begin{equation}
\label{eq:cov}
\cov(U^\top U, \det V) = \cov\{\E(U^\top U \mid V), \det V\} = -\cov(\tr V, \det V).
\end{equation}
According to \citet[][p.~112]{muirhead2005aspects}, the eigenvalues of $V \sim \be(\omega_1,\omega_2)$ are distributionally equivalent to the eigenvalues of $A(A+B)^{-1}$, where $A \sim \wish(2\omega_1, I_p)$ and $B \sim \wish(2\omega_2,I_p)$, independent.  Since $\tr V$ and $\det V$ are both coordinate-wise increasing functions of these eigenvalues, it follows from the main result of \citet[][Sec.~5]{dykstra1978positive} that $\cov(\tr V, \det V) \geq 0$.  This, along with \eqref{eq:cov}, completes the proof.  
\end{proof}

\begin{proof}[Proof of Theorem~\ref{thm:bpv}]
Here integrals shall be carried out in the form of exterior products of differentials, which we denote as $(d\mu)$, etc.  Use of exterior products leads to simpler change of variable formulas than those offered by traditional Jacobians.  The changes of variable used below, and the corresponding exterior products, can be found in \citet[][Chap.~2]{muirhead2005aspects}.  

Let $F^\star \sim \Pi^\star$ be the random measure that characterizes the absolutely continuous, rotation-invariant, location-scale family $\Pi_{\mu,\sigma}$, and let $f^\star$ denote its Radon--Nikodym derivative with respect to Lebesgue measure on $\RR^p$.  By Fubini's theorem, 
\begin{align*}
m_{\Pi,p+1}(x_1,\ldots, x_{p+1}) & = \int \Bigl[ \int_{\RR^p \times \TT_p} \Bigl\{ \prod_{i=1}^{p+1} (\det \sigma)^{-1} f^\star(\sigma^{-1}(x_i-\mu)) \Bigr\} \,d\pi_L(\mu,\sigma) \Bigr] \,d\Pi^\star(f^\star) \\
& = \int \Bigl[ \int_{\RR^p \times \TT_p} \Bigl\{ \prod_{i=1}^{p+1} f^\star(\sigma^{-1}(x_i-\mu)) \Bigr\} (\det\sigma)^{-(p+1)} \prod_{i=1}^p \sigma_{ii}^{-i} \, (d\mu)(d\sigma) \Bigr] \,d\Pi^\star(f^\star)\\
& = \int I(f^\star) \, d\Pi^\star(f^\star)
\end{align*}
where $I(f^\star)$ is the integral over $\RR^p \times \TT_p$ inside the square brackets above.  A change of variable $\tau = \sigma^{-1}$ implies $\tau$ ranges over $\TT_p$, $\sigma_{ii} = \tau_{ii}^{-1}$, $\det\sigma = (\det\tau)^{-1}$, and $(d\sigma) = (\det\tau)^{-(p+1)}\,(d\tau)$.  Therefore, 
\[ I(f^\star) = \int_{\RR^p \times \TT_p} \Bigl\{ \prod_{i=1}^{p+1} f^\star(\tau(x_i-\mu)) \Bigr\} \prod_{i=1}^p \tau_{ii}^{i} \, (d\mu)(d\tau). \]
Because of rotation-invariance, for any orthogonal matrix $\eta$, the random variables $I(f^\star)$ and $I(f^\star_{0,\eta})$ are identical in distribution.  Let $H$, with $dH(\eta) = (\eta^\top d\eta)/c_p$, denote the Haar measure on $\OO_p$, the space of $p \times p$ orthogonal matrices.  Then we must have 
\begin{align*}
m_{\Pi, p + 1}(x_1, \ldots, x_{p + 1}) & = \int_{\OO_p} \int I(f^\star_{0,\eta}) \, d\Pi^\star(f^\star) \, dH(\eta)\\ 
& = \int  \int_{\RR^p \times \TT_p \times \OO_p} c_p^{-1} \Bigl\{ \prod_{i=1}^{p+1} f^\star(\eta\tau(x_i - \mu)) \Bigr\} \prod_{i=1}^p \tau_{ii}^{i} (d\mu)(d\sigma) (\eta^\top d\eta) d\Pi(f^\star)\\
& = \int J(f^\star) \, d\Pi(f^\star)
\end{align*}
where $J(f^\star)$ is the inner integral above.
If we let $\nu = \eta\tau$, then $\nu$ ranges over the space $\mathbb{G}_p$ of $p \times p$ non-singular matrices, $\det\tau = |\det\nu|$, and $(d\nu) = \prod_{i=1}^p \tau_{ii}^{i-1}(d\tau)(\eta^\top d\eta)$.  Therefore, 
\[ J(f^\star) = c_p^{-1} \int_{\RR^p \times \mathbb{G}_p} \Bigl\{ \prod_{i=1}^{p+1} f^\star(\nu(x_i-\mu)) \Bigr\} |\det\nu| \, (d\mu)(d\nu). \]
Note that $(\mu,\nu)$ effectively ranges over $\RR^{p \times (p+1)}$, the $(p+1)$-fold product of $\RR^p$.  Make a final change of variable, $z_i = \nu(x_i-\mu)$, $i =1,\ldots,p+1$.  The inverse transformation is given by $\nu = \tilde z \tilde x^{-1}$, $\mu = x_{p+1} - \tilde x \tilde z^{-1} z_{p+1}$, where $\tilde x$ is as in the statement of the theorem and, likewise, $\tilde z$ is the $p \times p$ matrix with columns $\tilde z_i = z_i - z_{p+1}$.  Therefore, the Jacobian equals $|\det \tilde z||\det \tilde x|^{-(p-1)}$ and so 
\[ J(f^\star) = c_p^{-1} \int_{\RR^{p \times (p+1)}} \Bigl\{ \prod_{i=1}^{p+1} f^\star(z_i) \Bigr\} |\det \tilde x|^{-p} \,d(z_1,\ldots,z_{p+1}) = c_p^{-1} |\det\tilde x|^{-p}, \]
since $\int f^\star(z_i) \,dz_i = 1$ with $\Pi^\star$-probability~1 for each $i \in 1, \ldots, p+1$.  The claim \eqref{eq:pred} now follows immediately since $J(f^\star)$ is constant in $f^\star$.
\end{proof}

\begin{lemma}
\label{lem:rotation}
For $F^\star = \int \nm(u,v) \,d\Psitilde(u,v)$ with $\Psitilde \sim \dpp(\alpha,\Psi)$ and any $\eta \in \OO_p$, both $F^\star$ and $F_{0,\eta'}^\star$ have the same distribution.  
\end{lemma}

\begin{proof}
For $\eta \in \OO_p$, $d\nm(\eta x \mid u, v) = d\nm(x \mid \eta^\top u, \eta^\top v \eta)$ and, therefore, $F_{0,\eta}^\star = \int \nm(u,v) \,d\Psitilde_\eta(u,v)$, with $\Psitilde_\eta \sim \dpp(\alpha,\Psi_\eta)$, where $\Psi_\eta$ denotes the law of $(U_\eta,V_\eta) = (\eta^\top U, \eta^\top V\eta)$ when $(U,V) \sim \Psi$.  But if $V \sim \be(\omega_1,\omega_2)$, then also $V_\eta \sim \be(\omega_1,\omega_2)$ \citep[][Exercise 3.22d]{muirhead2005aspects} and if $U \mid V \sim \nm(0,I_p-V)$, then $U_\eta \mid V_\eta \sim \nm(0,I_p-V_\eta)$.  Therefore, by construction of $\Psi$, we have $\Psi_\eta = \Psi$ and, hence, $F^\star$ and $F_{0,\eta}^\star$ have the same distribution.  
\end{proof}

\begin{lemma}
\label{lem:inv}
Let $\Pi = \int \Pi_{\mu, \sigma}d\pi_L(\mu, \sigma)$ where $\{\Pi_{\mu, \sigma}, (\mu, \sigma) \in \RR^p \times \mathbb{T}_p\}$ is an absolute continuous, location-scale family. Then, for any $a \in \RR^p$, any $p\times p$ non-singular matrix $S$, any integer $n \ge p + 1$ and any $x_1, \ldots, x_n \in \RR^p$, $m_{\Pi, n}(a + Sx_1, \ldots, a + Sx_n) = |\det S|^{-(n - 1)} \cdot m_{\Pi, n}(x_1, \ldots, x_n)$. Consequently, the Bayes factor $B$ for \eqref{eq:null}, \eqref{eq:alt} is invariant under location and scale transformations of the data.
\end{lemma}

\begin{proof}
Let $F^\star \sim \Pi^\star$ denote the characterizing random measure of $\{\Pi_{\mu, \sigma}: \mu \in \RR^p, \sigma \in \mathbb{R}_p\}$ with Lebesgue density $f^\star$. As in the proof of Theorem \ref{thm:bpv} we can write $m_{\Pi, n}(x_{1:n}) = \int J( x_{1:n}\, |\, f^\star) d\Pi^\star(f^\star)$ where
\begin{align*}
J(x_{1:n}\mid f^\star) & = c_p^{-1}\int_{\RR^p \times \mathbb{T}_p \times \OO_p}  \left\{\prod_{i = 1}^n f^\star(\tau(x_i - \mu))\right\} \prod_{i = 1}^p \tau_{ii}^{n - p + i - 1} (d\mu)(d\tau)(\eta'd\eta)\\
& = c_p^{-1} \int_{\RR^p \times \mathbb{G}_p} \left\{\prod_{i = 1}^n f^\star(\nu(x_i - \mu))\right\} |\det \nu |^{n - p} (d\mu)(d\nu).
\end{align*}
Let $a + Sx_{1:n}$ denote the transformed data $(a + S x_1, \ldots, a  + S x_n)$. Then, 
\[
J(a + Sx_{1:n}\mid f^\star) = c_p^{-1}|\det S|^{-(n-p)} \int_{\RR^p \times \mathbb{G}_p}  \left\{\prod_{i = 1}^nf^\star(\tilde \nu(x_i - \tilde \mu))\right\} |\det \tilde\nu|^{n - p} (d\mu)(d\nu),
\]
with change of variables $\tilde \mu = S^{-1}(\mu - a)$ and $\tilde \nu = \nu S$. The Jacobian of this transformation is $(d\mu)(d\nu) = |\det S|^{-(p - 1)}(d\tilde \mu)(d\tilde \nu)$ and hence $J(a + S x_{1:n} | f^\star) = |\det S|^{-(n - 1)}J(x_{1:n} | f^\star)$.
\end{proof}

\section{Matrix F distribution}
\label{apndx:matF}
Let $S = V^\top V$ be the Cholesky factorization of $S$ where $V$ is an upper triangular matrix (as would be returned by the function call {\tt V = chol(S)} in R). A random draw of $\Sigma \sim \fdist(\nu - (p-1),\nu, S)$ could be obtained by setting $\Sigma = (UU^\top)^{-1}$ where $U = V^{-1}\Omega^{-1}\Lambda$ and $\Omega$ and $\Lambda$ are upper triangular matrices obtained as:
\begin{equation}
\begin{split}
\Omega^2_{jj} \sim \chi^2_{\nu - j + 1}, 1 \le j \le p;\quad\Omega_{ij} \sim \nm(0,1), 1 \le i < j \le p;\quad \Omega_{ij} = 0~\mbox{otherwise},\\
\Lambda^2_{jj} \sim \chi^2_{\nu - p + 1}, 1 \le j \le p;\quad\Lambda_{ij} \sim \nm(0,1), 1 \le i < j \le p; \quad \Lambda_{ij} = 0~\mbox{otherwise}.
\end{split}
\end{equation}
Notice that $U$ is upper triangular and one could write $\Sigma = \sigma \sigma^\top$ with $\sigma = (U^\top)^{-1}$ being lower triangular, a convention we have used throughout the paper. By introducing another upper triangular matrix $\Delta = \Omega^{-1}\Lambda$, we can express the probability density function of $\Sigma$ as
\[
p(\Sigma) = \frac{\Gamma_p(\nu)}{\Gamma_p(\nu/2)^2 |S|^{(p+1)/2}} \times \frac{|\Delta|^{\nu+p+1}}{|I_p + \Delta^\top \Delta|^\nu}.
\]
This formula is obtained by simplifying the general formulas derived in \citet{mulder2018matrix}.

\bibliographystyle{chicago} 
\bibliography{mybib}

\end{document}